\documentclass[10pt]{article}
\usepackage{amsmath}
\usepackage{amsthm}
\usepackage{amssymb}
\usepackage{amscd}
\usepackage{relsize}
\usepackage{enumerate}

\newtheorem{theorem}{Theorem}

\newcommand{\CC}{\mathbb{C}}
\newcommand{\RR}{\mathbb{R}}
\newcommand{\ZZ}{\mathbb{Z}}

\newcommand{\Hc}{\mathcal{H}}

\newcommand{\Ac}{\mathcal{A}}

\newcommand{\Lc}{\mathcal{L}}
\newcommand{\rank}{\operatorname{rank}}

\DeclareMathOperator*{\esup}{ess\,sup}
\DeclareMathOperator*{\einf}{ess\,inf}
\newcommand{\espan}{\operatorname{span}}

\title{{\bf Regular generalized sampling in $T$-invariant subspaces of a Hilbert space}}
\author{
{\bf Antonio~G. Garc\'{\i}a}\thanks{E-mail:\texttt{agarcia@math.uc3m.es}}, \, 
{\bf Mar\'{\i}a~J. Mu\~noz-Bouzo}\thanks{E-mail:\texttt{mjmunoz@mat.uned.es}}\,
{\bf  and \,G. P\'erez-Villal\'on}\thanks{E-mail:\texttt{gperez@euitt.upm.es}}
}
\date{}
\begin{document}
\maketitle
\begin{itemize}
\item[*] Departamento de Matem\'aticas, Universidad Carlos III de Madrid,
 Avda. de la Universidad 30, 28911 Legan\'es-Madrid, Spain.
\item[\dag] Departamento de Matem\'aticas Fundamentales, U.N.E.D., Senda del Rey 9, 28040 Madrid, Spain.
\item[\ddag] Departamento de Matem\'atica Aplicada a las Tecnolog\'{\i}as de la Informaci\'on y las Comunicaciones, E.T.S.I.T., U.P.M.,
 Avda. Complutense 30, 28040 Madrid, Spain.
\end{itemize}
\begin{abstract}
A regular generalized sampling theory in some structured $T$-invariant subspaces of a Hilbert space $\Hc$, where $T$ denotes a bounded invertible operator in $\Hc$, is established in this paper. This is done by walking through the most important cases which generalize the usual sampling settings.
\end{abstract}
{\bf Keywords}: $T$-invariant subspaces; Dual frames; Riesz sequences; Sampling expansions.

\noindent{\bf AMS}: 42C15; 94A20.
\section{Statement of the problem}
\label{section1}
The aim of this paper is to establish a regular generalized sampling theory in some structured $T$-invariant subspaces of an abstract separable Hilbert space $\Hc$, where $T$ denotes a bounded invertible operator on $\Hc$. Concretely, for a fixed $a\in \Hc$ these subspaces $\Ac_a$ are constructed by using a representation $h\mapsto \Pi(h)$ of a discrete LCA group $H$ (with additive notation) into the space of bounded invertible operators on $\Hc$ as
\[
\Ac_a=\Big\{\sum_{h\in H} \alpha_h\,\Pi(h) a\,:\, \{\alpha_h\}_{h\in H} \in \ell^2(H)\Big\}\,.
\]
The vector $a$ is called the generator of $\Ac_a$. The most important cases are those related with the representation of the groups $\ZZ$ or $\ZZ_N$ given by $n\mapsto T^n$ which yields the $\Ac_a$-subspaces
\[
\Big\{\sum_{n\in \ZZ} \alpha_n\,T^na\,:\, \{\alpha_n\}_{n\in \ZZ} \in \ell^2 (\ZZ)Ê\Big\}\quad \text{or}\quad \Big\{\sum_{n=0}^{N-1} \alpha(n)\,T^na\,:\, \{\alpha(n)\}_{n=0}^{N-1}\in \CC^N \Big\}\,.
\]
In the last case $T^Na=a$ and the cyclic group $\ZZ_N$ is represented on $\Ac_a$. Consequently, we have to describe the involved generalized samples and to exhibit the look of the obtained sampling formulas. Concerning the samples, consider $s$ fixed elements $b_j\in \Hc$, $j=1, 2,\dots,s$, that do not necessarily belong to $\Ac_a$. We define for each $x\in \Ac_a$  
\begin{equation}
\label{gs}
\Lc_jx(h):=\langle x, \Pi^*(-h)b_j\rangle_\Hc, \quad h\in H\,,
\end{equation}
where $\Pi^*(-h)$ denotes the adjoint operator of $\Pi(-h)$, and we restrict ourselves to the sequence of samples taken at a subspace $M$ of $H$, i.e., $\big\{\Lc_jx(m)\big\}_{m\in M;\,j=1,2,\ldots,s}$.  Regarding the sampling formulas they look like
\[
x=\sum_{j=1}^s \sum_{m \in M}\mathcal{L}_j x(m)\,\Pi(m)c_j \quad \text{in $\mathcal{H}$}\,,
\]
where $c_j\in \Ac_a$, $j=1,2,\ldots, s$, and the sequence $\big\{\Pi(m)c_j\big\}_{m\in M;\,j=1,2,\dots s}$ is a frame for $\mathcal{A}_a$. Thus, we obtain a stable recovery of any $x\in \Ac_a$ from the data sequence $\big\{\Lc_jx(m)\big\}_{m\in M;\,j=1,2,\ldots,s}$. In the alluded examples the samples are,  $\Lc_jx(rm)=\big\langle x, (T^*)^{-rm}b_j\big\rangle_\Hc$\,,\, $m\in \ZZ$ and $j=1, 2, \dots,s$, or $\Lc_jx(rn)=\langle x, (T^*)^{-rn}b_j \rangle_{\Hc}$\,, \, $n=0,1,\ldots,\ell-1$ and $j=1,2,\ldots, s$, respectively, where $r$ is the sampling period, a positive integer in the first case, or a positive integer such that $r|N$ and  $\ell=N/r$, in the second case. The above sampling formula reads
\[
x=\sum_{j=1}^s \sum_{m\in\mathbb{Z}}\mathcal{L}_j x(rm)\,T^{rm}c_j\quad \text{or}\quad x=\sum_{j=1}^s \sum_{m=0}^{\ell-1} \Lc_j x(rm)\, T^{rm} c_j\,.
\]
In the paper, the case of multiply generated subspaces is also considered since it entails new settings not covered by the one generator case. 
The used mathematical technique is very friendly: we express the given sequence of samples as frame coefficients in an auxiliary Hilbert space; the challenge is to obtain $T$-suitable dual frames yielding, via an isomorphism between the auxiliary Hilbert space and $\Ac_a$, the desired sampling formulas. The necessary  background on Riesz bases or frame theory in a separable Hilbert space can be found, for instance, in Ref. \cite{ole:16}.

The case where $T=U$ is an unitary operator in $\Hc$ has been studied in Refs. \cite{hector:14,garcia:16,garcia:15,garcia:17,michaeli:11,pohl:12}, and it generalizes averaged sampling in shift-invariant subspaces in $L^2(\RR)$; whenever $U$ is a shift operator, the samples given in \eqref{gs} are nothing but samples of a convolution operator, i.e., samples of a filtered version of the function itself.

The present sampling study in the $T$-invariant subspace $\Ac_a$ has a double motivation in the recent paper \cite{ole:17}: Firstly, any Riesz sequence $\{x_k\}_{k\in \ZZ}$ in $\Hc$ has a representation $\{T^k x_0\}_{k\in \ZZ}$ for a bounded and bijective operator $T: \overline{\espan} \{x_k\}_{k\in \ZZ}\rightarrow  \overline{\espan}\{x_k\}_{k\in \ZZ}$. Secondly, if $\{V^kg_0\}_{k\in \ZZ}$ with $V$ bounded is a dual frame of $\{T^kf_0\}_{k\in \ZZ}$ where $T$ is bounded and invertible, then $V=(T^*)^{-1}$. 

As it was mentioned in Ref.~\cite{ole:17}, the idea of considering frames (or Riesz) sequences of the form $\{T^na\}_{n\in\ZZ}$ (or $\{T^na\}_{n=0}^{N-1}$) is closely related with dynamical sampling (see, for instance, Refs.\cite{aldroubi1:17,aldroubi2:17}), although the indexing of a frame in the dynamical sampling context is different from the one used here; the group structure is crucial in the sequel.

The paper is organized as follows: in Section \ref{section2} we study the $\ZZ$-infinite case with a single generator, i.e., sampling formulas in $\Ac_a=\big\{\sum_{n\in \ZZ} \alpha_n\,T^na\,:\, \{\alpha_n\}_{n\in \ZZ} \in \ell^2 (\ZZ)Ê\big\}$; in Section \ref{section3} we present the $\ZZ$-infinite case with multiple generators, i.e., sampling formulas in $\mathcal{A}_{\mathbf{a}}=\Big\{\sum_{l=1}^L \sum_{n\in\mathbb{Z}}\alpha_n^l\,T^na_l:\, \{\alpha_n^l\}_{n\in \mathbb{Z}} \in \ell^2(\mathbb{Z});\, l=1, 2.\dots, L \Big\}$; Sections \ref{section4} and \ref{section5} are devoted to the finite case where both subspaces $\Ac_a$ or $\mathcal{A}_{\mathbf{a}}$ are finite dimensional; finally, in Section \ref{section6} the abstract case associated with an LCA group is exhibited. The  mathematical development followed along the paper  shares some patterns appearing in the case of an unitary operator $U$ studied in previous works (see \cite{hector:14,garcia:16,garcia:15,garcia:17}); some proofs which mimic similar results will be simply referred. Putting all these cases together can help to survey the intrinsic nature of these sampling problems and their relationships. As the sections only include the essential sampling theory, they are accompanied with a pertinent list of notes and remarks enlightening the topic.

\section{The $\mathbb{Z}$-infinite case with a single generator}
\label{section2}
Let $T:\Hc \rightarrow \Hc$ be a bounded invertible operator defined in a separable Hilbert space $\Hc$. For a fixed $a\in \Hc$ we consider the T-invariant subspace $\Ac_a$ in $\Hc$ defined as $\Ac_a:=\overline{\espan} \{T^na\,:\, n\in \ZZ\}$. In case the sequence $\{T^na\}_{n\in \ZZ}$ is a {\em Riesz sequence} in $\Hc$, i.e., a {\em Riesz basis} for $\Ac_a$, this subspace can be expressed as
\[
\Ac_a=\Big\{\sum_{n\in \ZZ} \alpha_n\,T^na\,:\, \{\alpha_n\}_{n\in \ZZ} \in \ell^2 (\ZZ)Ê\Big\}\,.
\]
The subspace $\mathcal{A}_a$ is the image of the usual Hilbert space $L^2(0, 1)$ by means of the isomorphism
\[
\begin{array}[c]{ccll}
\mathcal{T}_{T,a}: & L^2(0, 1) & \longrightarrow & \mathcal{A}_a\\
       & F=\displaystyle{\sum_{n\in \mathbb{Z}} \alpha_n\,{\rm e}^{2\pi {\rm i} nw} } & \longmapsto & x=\displaystyle{\sum_{n\in \mathbb{Z}} \alpha_n\,T^na}\,.
\end{array}
\]
It is easy to check that this isomorphism $\mathcal{T}_{T,a}$ satisfies the {\em $T$-shifting property}
\begin{equation}
\label{shifting1}
\mathcal{T}_{T,a}\big(F {\rm e}^{2\pi {\rm i} mw} \big)=T^m \big(\mathcal{T}_{T,a} F\big)\quad \text{for any $F\in L^2(0, 1)$ and $m\in \ZZ$}\,.
\end{equation}

\subsubsection*{An expression for the samples}
We start by stating the used data to recover any $x\in \Ac_a$.
Given $s$ fixed elements $b_j\in \Hc$  (which do not necessarily belong to $\Ac_a$) and the {\em sampling period} $r$, an integer $r\geq 1$, we define the {\em generalized samples} 
$\{\Lc_jx(rm)Ê\}_{m\in \mathbb{Z};\,j=1,2,\ldots,s}$ of any $x\in \Ac_a$ as 
\begin{equation}
\label{samples1}
\Lc_jx(rm):=\big\langle x, (T^*)^{-rm}b_j\big\rangle_\Hc\,,\quad \text{ $m\in \ZZ$\, and \, $j=1, 2, \dots,s$}\,,
\end{equation}
where $T^*$ denotes the (invertible) adjoint operator of $T$.

\noindent For $x=\sum_{n\in \mathbb{Z}} \alpha_n\,T^na$ we obtain the following expression for its samples
\begin{equation}
\label{samples2}
\begin{split}
\mathcal{L}_j x(rm)&=\big \langle \sum_{n\in \mathbb{Z}} \alpha_n\,T^na, (T^*)^{-rm}b_j\big\rangle_\Hc=\sum_{n\in \mathbb{Z}}\alpha_n\,\overline{\big\langle (T^*)^{-rm}b_j, T^na \big \rangle} _\Hc\\
&=\Big\langle F, \sum_{n\in \mathbb{Z}} \langle (T^*)^{n-rm}b_j, a\rangle_\mathcal{H}\, {\rm e}^{2\pi {\rm i} nw}\Big\rangle_{L^2(0, 1)}=
\big\langle F, g_j(w)\,{\rm e}^{2\pi {\rm i} rmw} \big\rangle_{L^2(0, 1)}\,,
\end{split}
\end{equation}
where the functions  $F(w)=\sum_{n\in \mathbb{Z}} \alpha_n\, {\rm e}^{2\pi {\rm i} nw}$ and $g_j(w)=\sum_{k\in \mathbb{Z}} \langle (T^*)^{k}b_j, a\rangle_\mathcal{H}\, {\rm e}^{2\pi {\rm i} kw}$, $j=1, 2, \dots,s$, belong to $L^2(0, 1)$. 

\medskip

Thus the stable recovery of $F\in L^2(0, 1)$ (and consequently of $x=\mathcal{T}_{T,a} F\in \Ac_a$) from the sequence of generalized samples $\{\Lc_jx(rm)Ê\}_{m\in \mathbb{Z};\,j=1,2,\ldots,s}$ depends on whether the sequence $\{g_j(w)\,{\rm e}^{2\pi {\rm i} rmw}\}_{m\in \mathbb{Z};\,j=1,2,\ldots,s}$ forms a {\em frame} for $L^2(0, 1)$. Moreover, in order to derive an associated sampling formula we also need to know  {\em dual frames} having the same structure (in order to apply the $T$-shifting property \eqref{shifting1}).

\medskip

For the first question, consider the $s\times r $ matrix of functions in $L^2(0,1)$
\begin{equation}
\label{Gmatrix1} 
\mathbb{G}(w):=
\begin{pmatrix} g_1(w)& g_1(w+\frac{1}{r})&\cdots&g_1(w+\frac{r-1}{r})\\
g_2(w)& g_2(w+\frac{1}{r})&\cdots&g_2(w+\frac{r-1}{r})\\
\vdots&\vdots&&\vdots\\
g_s(w)& g_s(w+\frac{1}{r})&\cdots&g_s(w+\frac{r-1}{r})
\end{pmatrix}= 
\bigg(g_j\Big(w+\frac{k-1}{r}\Big)\bigg)_{\substack{j=1,2,\ldots,s \\ k=1,2,\ldots, r}}
\end{equation}
and its related constants 
\[
\alpha_{\mathbb{G}}:=\einf_{w \in (0,1/r)}\lambda_{\min}[\mathbb{G}^*(w)\mathbb{G}(w)]\,;\quad
\beta_{\mathbb{G}}:=\esup_{w \in (0,1/r)}\lambda_{\max}[\mathbb{G}^*(w)\mathbb{G}(w)]\,.
\]
As usual, the symbol $*$ denotes the transpose conjugate matrix and $\lambda_{\min}$ (respectively $\lambda_{\max}$) the smallest (respectively the largest) eigenvalue of the positive semidefinite matrix 
$\mathbb{G}^*(w)\mathbb{G}(w)$. 

A characterization of the sequence $\big\{ g_j(w)\, {\rm e}^{2\pi {\rm i} rmw} \big\}_{m\in \mathbb{Z};\,  j=1,2,\dots, s}$ as a complete system, Bessel sequence, frame or Riesz basis for $L^2(0,1)$ is well known (see, for instance, Refs.~\cite{garcia:06,garcia:08}). In particular,

\medskip 

{\em The sequence $\{g_j(w)\,{\rm e}^{2\pi {\rm i} rmw}\}_{m\in \mathbb{Z};\,j=1,2,\ldots,s}$ is a frame for $L^2(0,1)$ if and only if $0<\alpha_{\mathbb{G}}\le \beta_{\mathbb{G}}<\infty$. In this case, the optimal frame bounds are $\alpha_{\mathbb{G}}/r$ and  $\beta_{\mathbb{G}}/r$.}

\medskip

For the second question, the existence of dual frames of $\{g_j(w)\,{\rm e}^{2\pi {\rm i} rmw}\}_{m\in \mathbb{Z};\,j=1,2,\ldots,s}$ with its same structure, choose functions $h_j$ in $L^\infty(0, 1)$, $j=1,2,\dots,s$, such that
\begin{equation}
\label{dual1}
\big(h_1(w), h_2(w), \dots, h_s(w)\big) \mathbb{G}(w)=\big(1, 0, \dots, 0\big) \quad \text{a.e. in $(0, 1)$}\,.
\end{equation}
In \cite{garcia:06} it was proven  that

\medskip

{\em The sequence $\{r \overline{h_{j}(w)}\,{\rm e}^{2\pi {\rm i} rmw}\}_{m\in\mathbb{Z};\,j=1,2,\ldots,s}$, where the functions $h_j$ in $L^\infty(0, 1)$ satisfy \eqref{dual1}, is a dual frame of the sequence 
$\{g_{j}(w)\,{\rm e}^{2\pi {\rm i} rmw}\}_{m\in\mathbb{Z};\,j=1,2,\ldots,s}$ in $L^2(0, 1)$.}

\medskip

In other words, taking into account the expression for the samples \eqref{samples2},  for any $F\in L^2(0, 1)$ we have   the expansion
\begin{equation}
\label{expansion1}
F=\sum_{j=1}^s \sum_{m\in \mathbb{Z}} \mathcal{L}_j x(rm)\,r \overline{h_j(w)}\,{\rm e}^{2\pi {\rm i}rmw}\quad \text{in $L^2(0, 1)$}\,.
\end{equation}

\medskip

Concerning to the existence of the functions $h_j$, $j=1, 2, \dots, s$, satisfying \eqref{dual1}, consider the first row of the $r\times s$ {\em Moore-Penrose pseudo-inverse} $\mathbb{G}^\dag(w)$ of $\mathbb{G}(w)$ given by $\mathbb{G}^\dag(w)=\big[\mathbb{G}^*(w)\,\mathbb{G}(w)\big]^{-1}\,\mathbb{G}^*(w)$. Its entries are essentially bounded in $(0, 1)$ since the functions $g_j$, $j=1,2,\dots, s$, 
and $\det^{-1}\big[\mathbb{G}^* (w)\, \mathbb{G}(w)\big]$ are essentially bounded in $(0, 1)$, and \eqref{dual1} trivially holds. In fact, all the possible solutions of \eqref{dual1} are given by the first row of the $r\times s$ matrices given by
\[
\mathbb{H}(w):=\mathbb{G}^\dag(w)+\mathbb{U}(w)\big[\mathbb{I}_s-\mathbb{G}(w)\mathbb{G}^\dag(w)\big]\,,
\]
where $\mathbb{U}(w)$ denotes any $r\times s$ matrix with entries in $L^\infty(0, 1)$, and $\mathbb{I}_s$ is the identity matrix of order $s$.
 
\subsubsection*{A regular sampling formula in $\Ac_a$}

Given $x=\mathcal{T}_{T,a} F\in \mathcal{A}_a$, applying the isomorphism $\mathcal{T}_{T,a}$ to the expansion \eqref{expansion1} for $F$ one obtains the sampling expansion
\[
\begin{split}
x&=\sum_{j=1}^s \sum_{m\in \mathbb{Z}} \mathcal{L}_j x(rm)\, \mathcal{T}_{T,a}\big[r \overline{h_j(\cdot)}\,{\rm e}^{2\pi {\rm i}rm\,\cdot}\big]=
\sum_{j=1}^s \sum_{m\in \mathbb{Z}} \mathcal{L}_j x(rm)\, T^{rm} \big[\mathcal{T}_{T,a}(r\overline{h}_j) \big]\\
&=\sum_{j=1}^s \sum_{m\in \mathbb{Z}} \mathcal{L}_j x(rm)\, T^{rm} c_{j,h} \quad \text{in $\mathcal{H}$}\,,
\end{split}
\]
where $c_{j, h}:=\mathcal{T}_{T,a}(r\overline{h}_j)\in \mathcal{A}_a$, $j=1,2,\dots, s$, and we have used the $T$-shifting property \eqref{shifting1}. Besides, the sequence 
$\big\{T^{rm} c_{j,h}  \big\}_{m\in\mathbb{Z};\,j=1,2,\ldots,s}$ is a frame for $\mathcal{A}_a$. In fact, the following result holds:  
\begin{theorem}
\label{regular1}
For any $x\in \Ac_a$ consider the sequence of samples $\{\mathcal{L}_jx(rm)\}_{m\in \mathbb{Z};\,j=1,2,\dots, s}$ defined in \eqref{samples1}.  Assume that the functions $g_j$, $j=1,2,\dots,s$, given in \eqref{samples2} belong to $L^\infty(0, 1)$, and consider the associated $\mathbb{G}(w)$ matrix given in \eqref{Gmatrix1}. The following statements are equivalent:
\begin{enumerate}[(a)]
\item The constant $\alpha_{\mathbb{G}}>0$.
\item There exists a vector $\big(h_1(w), h_2(w), \dots, h_s(w)\big)$ with entries in $L^\infty(0, 1)$ and satisfying 
\[
\big(h_1(w), h_2(w), \dots, h_s(w)\big) \mathbb{G}(w)=\big(1, 0, \dots, 0\big) \quad \text{a.e. in $(0, 1)$}\,.
\]
\item There exist $c_{j,h}\in \mathcal{A}_a$, $j=1,2, \dots, s$, such that the sequence 
$\big\{T^{rm}c_{j,h}\big\}_{m\in \mathbb{Z};\,j=1,2,\dots, s}$ is a frame for $\mathcal{A}_a$, and for any $x\in \mathcal{A}_a$ the expansion
\begin{equation}
\label{sampling1}
x=\sum_{j=1}^s \sum_{m\in\mathbb{Z}}\mathcal{L}_j x(rm)\,T^{rm}c_{j,h} \quad \text{in $\mathcal{H}$}
\end{equation}
holds.
\item There exists a frame $\big\{C_{j,m}\big\}_{m\in \mathbb{Z};\,j=1,2,\dots, s}$ for $\mathcal{A}_a$ such that, for each $x\in \mathcal{A}_a$ the expansion
\[
x=\sum_{j=1}^s \sum_{m\in\mathbb{Z}}\mathcal{L}_{j} x(rm)\,C_{j,m} \quad \text{in $\mathcal{H}$}
\]
holds.
\end{enumerate}
\end{theorem}
\begin{proof}
First notice that the equivalence between the spectral and Frobenius norms \cite{horn:99} proves that the functions $g_j$, $j=1, 2, \dots, s$, belong to $L^\infty (0,1)$ if and only if  $\beta_{\mathbb{G}}<\infty$. We have already proved that $(a)$ implies $(b)$, and that $(b)$ implies $(c)$. Obviously, $(c)$ implies $(d)$.
The proof concludes as that of Theorem 3.1 in \cite{hector:14}.
\end{proof}
\subsubsection*{The filter-bank approach}
The generalized samples $\{(\mathcal{L}_{j}x)(rm)\}_{m\in \mathbb{Z};\,j=1,2,\dots, s}$ defined in \eqref{samples1} can be seen as the output of an analysis filter-bank. To check this, notice that for $x=\sum_{n\in \mathbb{Z}} \alpha_{n}T^n a \in \mathcal{A}_{a}$ and $\mathcal{L}_{j}x$ (defined on $\ZZ$) we have
\[
\mathcal{L}_{j}x(m)= (\alpha \ast \mathbf{h}_{j})(m), \,\,\, m\in \ZZ\,,\quad \text{where}\quad \mathbf{h}_{j}(n):=\left\langle a, (T^*)^{-n} b_{j}\right\rangle_{\mathcal{H}},\,\,\, n\in \mathbb{Z}.
\]
Indeed,
$(\mathcal{L}_{j}x) (m)=\left\langle \sum_{n\in \mathbb{Z}} \alpha_{n}T^n a , (T^*)^{-m} b_{j}\right\rangle_{\mathcal{H}}=
 \sum_{n\in \mathbb{Z}} \alpha_{n} \left\langle a , (T^*)^{n-m} b_{j}\right\rangle_{\mathcal{H}}$.
Thus, the sequence of samples $\big\{(\mathcal{L}_{j}x)(rm)\big\}_{m\in \mathbb{Z};\,j=1,2,\dots, s}$ is the output of the {\em analysis filter-bank} with {\em downsampling} $r$
\[
\alpha \longmapsto \Big( \big\{\alpha \ast \mathbf{h}_{1}(rm)\big\}_{m\in\mathbb{Z}},\ldots,
\big\{\alpha \ast \mathbf{h}_{s}(rm)\big\}_{m\in\mathbb{Z}} \Big)\,.
\]
Any sequence $\alpha \in \ell^2(\mathbb{Z})$ can be recovered from these samples by using a {\em synthesis filter-bank}: 
\[
\big( y_{1}, y_2,\ldots,
y_{s} \big) \longmapsto \bigg\{\sum_{j=1}^s \sum_{m\in \mathbb{Z} }y_{j}(m) \mathbf{g}_{j}(n-mr)\bigg\}_{n\in \mathbb{Z}}\,,
\]
whenever the above analysis  and synthesis filter-banks
\[
\alpha \longmapsto \bigg\{\sum_{j=1}^s  \sum_{m\in \mathbb{Z} } (\alpha \ast \mathbf{h}_{j})(rm) \mathbf{g}_{j}(n-mr)\bigg\}_{n\in \mathbb{Z}}
\]
provides {\em perfect reconstruction}, i.e., for every $\alpha=\{\alpha_n\}_{n\in \ZZ}\in \ell^2(\mathbb{Z})$ we have
\begin{equation}
\label{PR}
\alpha_n = \sum_{j=1}^s  \sum_{m\in \mathbb{Z} } (\alpha \ast \mathbf{h}_{j})(rm) \mathbf{g}_{j}(n-mr)\,, \quad n \in \mathbb{Z}\,.
\end{equation} 
For instance, assuming that $\mathbf{h}_{j}, \mathbf{g}_{j}\in \ell^1(\ZZ)$  for $j=1,2,\dots,s$, perfect reconstruction holds
if and only if  $\mathbf{G}(z)\mathbf{H}(z)=\mathbb{I}_r$ in the torus $\mathbb{T}=\{z\in \CC \,:\,|z|=1\}$ (see \cite{bol:98,vetterli:98,strang:96}),
where the $r\times s$ matrix $\mathbf{G}(z)$ and the $s\times r$ matrix $\mathbf{H}(z)$ given, respectively, by
\[
\mathbf{H}(z)=\bigg( \sum_{m\in\mathbb{Z}} \mathbf{h}_{j}(rm-k) z^{-m}\bigg)_{\substack{j=1,2,\ldots,s \\ k=0,1,\ldots, r-1}}\text{and} \,\, \mathbf{G}(z)=\bigg( \sum_{m\in\mathbb{Z}} \mathbf{g}_{j}(rm+k) z^{-m}\bigg)_{\substack{k=0,1,\ldots,r-1 \\ j=1,2,\ldots, s}}
\] 
are the so called {\em polyphase matrices} associated to the above filter-banks. Equivalently, the sequences $\{\mathbf{\overline{h}}_j(rm-\cdot)\}_{m\in \mathbb{Z};\,j=1,2,\dots, s}$ and $\{\mathbf{g}_j(\cdot-rm)\}_{m\in \mathbb{Z};\,j=1,2,\dots, s}$ form a pair of dual frames in $\ell^2(\ZZ)$.

Summarizing, assuming perfect reconstruction on the above filter-banks we can recover any sequence $\alpha=\{\alpha_n\}_{n\in \ZZ}$ in $\ell^2(\mathbb{Z})$ from the generalized samples \eqref{samples1} by using formula \eqref{PR}. As a consequence, we can recover the corresponding $x\in \mathcal{A}_{a}$ by using $x=\sum_{n\in \mathbb{Z}} \alpha_{n}T^n a$.

\subsubsection*{Notes and remarks}
Some comments on the results appearing in this section are pertinent; most of them will be shared in next sections:
\begin{enumerate}
\item The fact of considering subspaces as $\Ac_a$ is reinforced by a result proved in Ref.\cite[Corollary 2.4]{ole:17}: Any Riesz sequence $\{x_k\}_{k\in \ZZ}$ in $\Hc$ has a representation $\{T^k x_0\}_{k\in \ZZ}$ for a bounded and bijective operator $T: \overline{\espan} \{x_k\}_{k\in \ZZ}\rightarrow  \overline{\espan}\{x_k\}_{k\in \ZZ}$.

\item In case the operator $T=U$ is unitary, the auto-covariance $\langle U^ka, a \rangle_\mathcal{H}$, $k\in \ZZ$, of the sequence $\{U^ka\}_{k\in \ZZ}$ admits the integral representation (see Ref.\cite{kolmogorov:41})
\[
\langle U^ka, a \rangle_\mathcal{H}=\frac{1}{2\pi}\int_{-\pi}^\pi {\rm e}^{{\rm i}k\theta}d\mu_a(\theta)\,,\qquad k\in\mathbb{Z}\,,
\]
where $\mu_a$ is  a positive Borel measure on $(-\pi, \pi)$ called the {\em spectral measure} of the sequence $\{U^ka\}_{k\in \ZZ}$. The spectral measure $\mu_a$ can be decomposed into an absolute continuous and a singular part as $d\mu_a(\theta)=\phi_a(\theta)d\theta+ d\mu^{s}_a(\theta)$ with respect to Lebesgue measure. A necessary and sufficient condition for the sequence $\{U^na\}_{n\in \mathbb{Z}}$ to be a Riesz sequence for $\mathcal{H}$ is that the singular part $\mu_a^{s}\equiv 0$ and the {\em spectral density} $\phi_a$ satisfies
\[
0< \einf_{\theta \in (-\pi, \pi)}{\phi_a(\theta)}\leq \esup_{\theta \in (-\pi, \pi)}{\phi_a(\theta)}<\infty \,;
\]
see, for instance, Ref.~\cite{hector:14}. In particular, whenever $U$ is the shift operator $f(u)\mapsto f(u-1)$ in $L^2(\RR)$, the above condition yields the classical condition (see, for instance, Ref.~\cite{ole:16})
\[
0< \einf_{\theta \in (-\pi, \pi)} \sum_{n\in \mathbb{Z}}|\widehat{\varphi}(\theta+2\pi n)|^2 \le \esup_{\theta \in (-\pi, \pi)}\sum_{n\in \mathbb{Z}}|\widehat{\varphi}(\theta+2\pi n)|^2 <\infty\,.
\]
As far as we know, a characterization of the sequence $\{T^na\}_{n\in \ZZ}$ as a Riesz sequence for $\Hc$ remains an open question; even when the operator $T$ is selfadjoint or normal.

\item In case the sequence $\{T^ka\}_{k\in \ZZ}$ is a frame sequence for $\Hc$, the operator $\mathcal{T}_{T,a}$ is bounded and surjective. The sampling formula \eqref{sampling1} still remains valid as a frame expansion in $\Ac_a$.

\item The choice of the generalized samples as in \eqref{samples1} is motivated by a result in Ref.\cite[Lemma 3.3]{ole:17}: If $\{V^kg_0\}_{k\in \ZZ}$ with $V$ bounded is a dual frame of $\{T^kf_0\}_{k\in \ZZ}$ where $T$ is bounded and invertible, then $V=(T^*)^{-1}$.

\item In Theorem \ref{regular1} it can be added the equivalent condition 
\[
\einf_{w \in (0,1/r)}\det [\mathbb{G}^*(w)\mathbb{G}(w)]>0\,.
\]
In case the $1$-periodic extension of the functions $g_j$, $j=1, 2, \dots, s$, are continuous on $\RR$, this condition reduces to say that $\rank \,\mathbb{G}(w)=r$ for all $w\in \RR$.

\item In the {\em overcomplete} setting we have that $s>r$. In case $r=s$, the frame condition $(c)$ in Theorem \ref{regular1} becomes a Riesz basis condition: There exist $r$ unique elements $c_j\in \Ac_a$, $j=1, 2, \dots, r$, such that the sequence $\big\{T^{rm}c_j\big\}_{m\in \mathbb{Z};\,j=1,2,\dots, r}$ is a {\em Riesz basis} for $\mathcal{A}_a$, and the sampling expansion \eqref{sampling1} holds. Moreover, due to the uniqueness of the coefficients in a Riesz basis expansion, the {\em interpolation property} 
$\mathcal{L}_{j'} c_j(rm)=\delta_{j,j'}\, \delta_{m,0}$, where $m\in \mathbb{Z}$ and $j, j'=1,2,\dots, r$, holds (for a similar result, see \cite[Corollary 3.3]{hector:14}).

\item Let us to take a closer look at the analyzing sequence $\big\{(T^*)^{-rm}b_j\big\}_{m\in \mathbb{Z};\,j=1,2,\dots,s}$ which appears in the definition of the generalized samples \eqref{samples1}. Having in mind \eqref{samples2} and the isomorphism $\mathcal{T}_{T,a}$, we have the inequalities
\[
\frac{\alpha_{\mathbb{G}}}{r} \|\mathcal{T}_{T,a}\|^{-2} \|x\|^2 \le \sum_{j=1}^s\sum_{m\in \mathbb{Z}} \big|\langle x, (T^*)^{-rm} b_j\rangle \big|^2 \le \frac{\beta_{\mathbb{G}}}{r} \|\mathcal{T}_{T,a}^{-1}\|^{2} \|x\|^2 \quad \text{for all $x\in \mathcal{A}_a$}\,.
\] 
The sequence $\big\{(T^*)^{-rm}b_j\big\}_{m\in \mathbb{Z};\,j=1,2,\dots, s}$ is not contained in $\mathcal{A}_a$ except for some particular cases such as whenever all $b_j\in \Ac_a$ and operator $T$ is selfadjoint or unitary. Therefore, as a consequence of the above inequalities, the sequence  
$\big\{(T^*)^{-rm}b_j\big\}_{m\in \mathbb{Z};\,j=1,2,\dots, s}$ is a {\em pseudo-dual frame} of $\big\{T^{rm}c_{j,h}\big\}_{m\in \mathbb{Z};\,j=1,2,\dots, s}$ in 
$\mathcal{A}_a$ (see Refs.~\cite{li:01,li:04}). In other words, denoting by $P_{\mathcal{A}_a}$ the orthogonal projection onto 
$\mathcal{A}_a$, we derive  that the sequence $\big\{P_{\mathcal{A}_a}\big((T^*)^{-rm}b_j\big)\big\}_{m\in \mathbb{Z};\,j=1,2,\dots, s}$ is a dual frame of 
$\big\{T^{rm}c_{j,h}\big\}_{m\in \mathbb{Z};\,j=1,2,\dots, s}$ in $\mathcal{A}_a$.

Whenever $r=s$, the sequence $\big\{(T^*)^{-rm}b_j\big\}_{m\in \mathbb{Z};\,j=1,2,\dots ,s}$ is, except for some particular cases, a {\em pseudo-Riesz basis} for $\mathcal{A}_a$.

\item The oversampling technique, i.e.  $s>r$, allows to obtain sampling reconstruction formulas with prescribed properties. See example below and Refs.\cite{hector:14,garcia:09}.
\end{enumerate}

Theorem \ref{regular1} comprises all the known results concerning average regular sampling  in shift-invariant spaces which appear in the mathematical literature. For a few selected references see, for instance, Refs.~\cite{aldroubi:02,ole:04,ole:05,garcia:06,garcia:08,kang:11,sun:03,unser:94,unser:98,walter:92}. 

\subsubsection*{An easy illustrative example}
In $\mathcal{H}=\ell^2(\mathbb{Z})$ we consider the unitary operator $Tx(\cdot)=x(\cdot-K)$, where $K\in \mathbb{N}$ denotes a fixed delay. Now, let $a\in \ell^2(\ZZ)$ such that the sequence $\{a(\cdot-nK)\}_{n\in \ZZ}$ forms a Riesz sequence for $\ell^2(\ZZ)$. According to note 2 above, this happens in case the sequence $\{\langle T^na, a\rangle_{\ell^2(\ZZ)}\}_{n\in \ZZ}\in \ell^2(\ZZ)$, which  in turn occurs, for instance, whenever $a \in \ell^1(\ZZ)$,  and the function having these Fourier coefficients should be essentially bounded above and below away from zero. Consider  the corresponding $T$-invariant subspace $\mathcal{A}_{a}$.

First assume $r=s=1$;  for a fixed $b\in \ell^2(\ZZ)$, for each $x\in \Ac_a$, the samples are $\mathcal{L}x(n)=\left\langle x, b(\cdot-nK) \right\rangle_{ \ell^2(\mathbb{Z})}, \,\, n\in \mathbb{Z}$. Thus, the auxiliary function $g$ reads
\[
g(w)=\sum_{n\in \mathbb{Z} } \left\langle b(\cdot+nK),a \right\rangle_{ \ell^2(\mathbb{Z})} e^{2\pi i n w}\in L^2(0,1)\,.
\]
Assuming that $g\in L^\infty(0,1)$, Theorem \ref{regular1} says that there exists a (unique) sequence $c\in \mathcal{A}_{a}$ such that  the expansion
\begin{equation*}
\label{for}
x(m)= \sum_{n\in \mathbb{Z} }  \left\langle x, b(\cdot-nK) \right\rangle_{ \ell^2(\mathbb{Z})}\, c(m-nK),\quad m\in \ZZ\,,
\end{equation*} 
holds for all $x\in \Ac_a$ if and only if $\einf_{w \in (0,1/r)} |g(w)|>0$. The sequence $\{c(\cdot-nK)\}_{n\in \mathbb{Z}}$ is a Riesz basis for 
$\mathcal{A}_{a}$ (see note 6 above), and $c(m)=\sum_{n\in \mathbb{Z} } \beta_{n} \, a(m-nK)$ where the coefficients  
$\beta_{n}$ are obtained from the Fourier expansion  $1/\overline{g(w)}= \sum_{n\in \mathbb{Z} } \beta_{n} e^{2\pi in w}$.

The choice of the delta sequence as $b$ in the above formula yields for each $x\in \mathcal{A}_{a}$
\begin{equation}
\label{for1}
x(m)= \sum_{n\in \mathbb{Z} }   x(nK)\, c(m-nK)\,, \quad m\in \ZZ\,.
\end{equation}
Furthermore, the sequence $c\in \mathcal{A}_{a}$ satisfies the interpolation property $c(nK)=\delta(n)$, $n\in \ZZ$.

\medskip

For $a=M_{p}$ the {\em central discrete B-spline} defined (assuming that $K$ is odd) by 
\[
M_{p}:=\underbrace{M_{1}\ast\ldots \ast M_{1}}_{\text{(p times)}}\,\, \text{  where }
M_{1}(n)=\begin{cases}1,&|n|\le (K-1)/2\\ 0,& |n|>(K-1)/2\end{cases}\,,
\] 
$\mathcal{A}_{a}$ becomes the subspace of discrete splines  with  nodes at $\{nK\}$ of order $p$ (see Refs.~\cite{zhe:00,zhe:02}). The B-spline $M_{p}(n)$ is even, positive, supported in $|n|\le p(K-1)/2$, and the sequence $\{M_{p}(\cdot-nK)\}_{n\in\mathbb{Z}}$ is a Riesz sequence for $\ell^2(\ZZ)$  \cite[Remark 3.3]{zhe:02}. 
Furthermore, the cosine polynomial 
\[
g(w)= \sum_{n\in \mathbb{Z} } M_{p}(nK) e^{2\pi i n w} 
\]
is strictly positive (see Ref.~\cite[Theorem 2.2]{zhe:00}); as a consequence, the interpolatory formula \eqref{for1} holds.  

A drawback is that the reconstruction function $c$ in \eqref{for1} is not compactly supported. A way to overcome this difficulty is by using the oversampling technique, i.e., taking $s>r$. For instance, consider $a=M_{p}, r=1$, but 
$s=2$ with $b_{1}=\delta$ and $b_{2}=\delta(\cdot-1)$.
Thus, we get $g_{1}(w)=G_{1}(e^{-2\pi i w})$ and $g_{2}(w)=G_{2}(e^{-2\pi i w})$ where
\[
G_{1}(z)=\sum_{n\in \mathbb{Z} } M_{p}(nK) z^n\quad  \text{and}\quad
G_{2}(z)=\sum_{n\in \mathbb{Z} } M_{p}(nK+1) z^n
\]
are Laurent polynomials in $z$. If $G_{1}$ and $G_{2}$ are coprime, one can find Laurent polynomials $H_{1}$ and $H_{2}$ such that $G_{1}(z)H_{1}(z)+G_{2}(z)H_{2}(z)=1$.
As a consequence,  the functions $h_{1}(w)=H_{1}(e^{-2\pi i w})$ and  $h_{2}(w)=H_{2}(e^{-2\pi i w})$, that satisfy condition $(b)$ in Theorem \ref{regular1}, yield reconstruction  sequences $c_{1}$ and $c_{2}$ in formula \eqref{for1}, which are compactly supported. For instance, taking $K=3$ and the cubic B-spline $M_{4}$ one gets $G_{1}(z)=4 z^{-1}+19+4 z,\,G_{2}(z)=10 z^{-1}+16 +z$. Euclid's algorithm gives $H_{1}(z)=-\frac{38}{243}z-\frac{5}{486}z^2$ and
$H_{2}(z)=\frac{79}{486}z+\frac{10}{243}z^2$ from which compactly supported sequences $c_{1}$ and $c_{2}$ are derived. 

The above oversampling rate is $s/r=2$. Less oversampling rates, as $s/r=(q+1)/q$, can be used; the prize to pay is a bigger size of the support for the reconstruction sequences. This study was done in Refs.\cite{garcia1:09,garcia2:09} in the most usual setting where $\mathcal{H}=L^2(\mathbb{R})$ and $Tf(t)=f(t-1)$.

\section{The  $\mathbb{Z}$-infinite case with multiple generators}
\label{section3}
The case of $L$ generators can be analogously handled. Indeed, consider the subspace generated by $\mathbf{a}:=\{a_1, a_2, \dots, a_L\} \subset \mathcal{H}$, i.e., 
$\mathcal{A}_{\mathbf{a}}:= \overline{{\rm span}}\big\{T^na_l,\;  n\in \mathbb{Z};\, l=1, 2.\dots, L\big\}$.
Assuming that the sequence $\{T^na_l\}_{n\in \mathbb{Z};\,l=1, 2, \dots, L}$ is a Riesz sequence for
$\mathcal{H}$,  the subspace $\mathcal{A}_{\mathbf{a}}$ can be expressed as
\[
\mathcal{A}_{\mathbf{a}}=\Big\{\sum_{l=1}^L \sum_{n\in\mathbb{Z}}\alpha_n^l\,T^na_l:\, \{\alpha_n^l\}_{n\in \mathbb{Z}} \in \ell^2(\mathbb{Z});\, l=1, 2.\dots, L \Big\}\,.
\]
The subspace $\mathcal{A}_{\mathbf{a}}$ is the image of the usual product Hilbert space $L_L^2(0, 1)$ by means of the isomorphism
\[
\begin{array}[c]{ccll}
\mathcal{T}_{T,\mathbf{a}}: & L_L^2(0, 1) & \longrightarrow & \mathcal{A}_\mathbf{a}\\
       & \mathbf{F}=\displaystyle{\sum_{l=1}^L\sum_{n\in \mathbb{Z}} \alpha_n^l\,{\rm e}^{2\pi {\rm i} nw}\mathbf{e}_l } & \longmapsto & x=\displaystyle{\sum_{l=1}^L\sum_{n\in \mathbb{Z}} \alpha_n^l\,T^na_l}\,,
\end{array}
\]
where $\{\mathbf{e}_l\}_{l=1}^L$ denotes the canonical basis for $\mathbb{C}^L$. This isomorphism $\mathcal{T}_{T,\mathbf{a}}$ satisfies, for each $\mathbf{F}=(F_1, F_2,\ldots,F_L)^\top$ in $L_L^2(0, 1)$, the {\em $T$-shifting property}
\begin{equation}
\label{shifting2}
\mathcal{T}_{T,\mathbf{a}}\big(\mathbf{F} {\rm e}^{2\pi {\rm i} mw} \big)=T^m \big(\mathcal{T}_{T,a} \mathbf{F}\big)\,,\quad m\in \ZZ\,.
\end{equation}

\subsubsection*{An expression for the samples}
For $x=\mathcal{T}_{T,\mathbf{a}} \mathbf{F} \in \mathcal{A}_{\mathbf{a}}$ we consider its samples $\{\mathcal{L}_jx(rm)\}_{m\in \mathbb{Z};\,j=1,2,\dots, s}$ given by \eqref{samples1}. In this case, assuming that $x=\sum_{l=1}^L\sum_{n\in \mathbb{Z}} \alpha_n^l\,T^na_l$ we have 
\[
\mathcal{L}_jx(rm)=\big \langle \sum_{l=1}^L\sum_{n\in \mathbb{Z}} \alpha_n^l\,T^na_l, (T^*)^{-rm}b_j\big\rangle_\Hc=\sum_{l=1}^L \big\langle F_l, g_j^l(w)\,{\rm e}^{2\pi {\rm i} rmw} \big\rangle_{L^2(0, 1)}\,,
\]
where $\displaystyle{F_l(w)=\sum_{n\in \mathbb{Z}} \alpha_n^l\, {\rm e}^{2\pi {\rm i} nw}}$ and $\displaystyle{g_j^l(w)=\sum_{k\in \mathbb{Z}} \langle (T^*)^{k}b_j, a_l\rangle_\mathcal{H}\, {\rm e}^{2\pi {\rm i} kw}}$ for $l=1, 2, \dots,L$ and $j=1, 2, \dots,s$. In other words, we have obtained the expression for the samples
\begin{equation}
\label{samples2bis}
\mathcal{L}_jx(rm)=\big\langle \mathbf{F}, \mathbf{g}_j(w)\,{\rm e}^{2\pi {\rm i} rmw} \big\rangle_{L_L^2(0, 1)}\,,
\end{equation}
where  the functions $\mathbf{F}=(F_1, F_2, \dots, F_L)^\top$  and
$\mathbf{g}_j(w)=(g_j^1(w), g_j^2(w), \dots,g_j^L(w))^\top$ belong to $L_L^2(0, 1)$ for $j=1, 2, \dots,s$. 

As in the one-generator case, the sequence $\big\{\mathbf{g}_j(w)\,{\rm e}^{2\pi {\rm i} rmw} \big\}_{m\in \mathbb{Z};\, j=1,2,\dots, s}$ should be a frame for $L_L^2(0, 1)$. Now, $\mathbb{G}$ turns out to be an $s\times rL$ matrix of  functions in $L^2(0, 1)$, namely
\[
\mathbb{G}(w)=\bigg(\mathbf{g}_j^\top\Big(w+\frac{k-1}{r}\Big)\bigg)_{\substack{j=1,2,\ldots,s \\ k=1,2,\ldots, r}}\,.
\]
Its corresponding related constants $\alpha_{\mathbb{G}}$ and $\beta_{\mathbb{G}}$ must also verify the necessary and sufficient condition $0<\alpha_{\mathbb{G}}\le \beta_{\mathbb{G}}<\infty$ (see \cite[Lemma 2]{garcia:08}).

Its dual frames having the same structure are $\big\{r\overline{\mathbf{h}_{j}(w)}\,{\rm e}^{2\pi {\rm i} rmw}\big\}_{m\in\mathbb{Z};\,j=1,2,\ldots,s}$, where the functions 
$\mathbf{h}_{j}$ form an $L\times s$ matrix 
$\mathbf{h}(w):=\big(\mathbf{h}_{1}(w), \mathbf{h}_{2}(w), \dots, \mathbf{h}_{s}(w)\big)$ with entries in $L^\infty(0, 1)$, and satisfying
\[
\Big(\mathbf{h}_{1}(w), \mathbf{h}_{2}(w), \dots, \mathbf{h}_{s}(w)\Big)\, \mathbb{G}(w)=\big( \mathbb{I}_L, \mathbb{O}_{L\times(r-1)L}\big)\quad \text{a.e. in $(0, 1)$}\,.
\]
In other words, any dual frame of  
$\big\{\mathbf{g}_{j}(w)\,{\rm e}^{2\pi {\rm i} rmw}\big\}_{m\in\mathbb{Z};\,j=1,2,\ldots,s}$ with its same structure is obtained by taking the first $L$ rows of the 
$rL\times s$ matrices given by
\[
\mathbb{H}(w):=\mathbb{G}^\dag(w)+\mathbb{U}(w)\big[\mathbb{I}_s-\mathbb{G}(w)\mathbb{G}^\dag(w)\big]\,,
\]
where $\mathbb{U}(w)$ denotes any $rL\times s$ matrix with entries in $L^\infty(0, 1)$. See \cite{garcia:08} for the details.

\subsubsection*{A regular sampling formula in $\Ac_\mathbf{a}$}

Since any $\mathbf{F} \in L_L^2(0, 1)$ can be expanded as
\[
\mathbf{F}=\sum_{j=1}^s \sum_{m\in \mathbb{Z}} \big\langle \mathbf{F}, \mathbf{g}_j(w)\,{\rm e}^{2\pi {\rm i}rmw} \big\rangle_{L_L^2(0, 1)}\, r\overline{\mathbf{h}_{j}(w)}\,{\rm e}^{2\pi {\rm i} rmw}\quad \text{in $L_L^2(0, 1)$}\,,
\]
having in mind \eqref{samples2bis}, the isomorphism $\mathcal{T}_{T,\mathbf{a}}$ and the $T$-shifting property \eqref{shifting2}, for each $x=\mathcal{T}_{T,\mathbf{a}} \mathbf{F}\in \mathcal{A}_{\mathbf{a}}$ give the sampling expansion
\begin{equation}
\label{sampling1bis}
x=\sum_{j=1}^s \sum_{m\in \mathbb{Z}} \mathcal{L}_j x(rm)\, T^{rm}\big[\mathcal{T}_{T,\mathbf{a}}(r\mathbf{h}_j) \big]
=\sum_{j=1}^s \sum_{m\in \mathbb{Z}} \mathcal{L}_j x(rm)\, T^{rm} c_{j,\mathbf{h}}\quad \text{in $\mathcal{H}$}\,,
\end{equation}
where $c_{j,\mathbf{h}}=\mathcal{T}_{T,\mathbf{a}}(r\mathbf{h}_j)\in \mathcal{A}_{\mathbf{a}}$, $j=1, 2, \dots, s$. The sequence 
$\{T^{rm} c_{j,\mathbf{h}}\}_{m\in\mathbb{Z};\,j=1,2,\ldots,s}$ is a frame for $\mathcal{A}_{\mathbf{a}}$. 

A similar characterization of the sampling formulas \eqref{sampling1bis}, analogous to that in Theorem \ref{regular1}, can be stated for this multiple generators setting.
\subsubsection*{Notes and remarks}
Next we include some specific comments on this section:
\begin{enumerate}
\item The use of several generators enriches the subspace $\mathcal{A}_{\mathbf{a}}$. Thus, multiply generators in the shift-invariant case leads to the multiwavelet setting. Multiwavelets lead to multiresolution analyses and fast algorithms just as scalar wavelets, but they have some advantages: they can have short support coupled with high smoothness and high approximation order, and they can be both symmetric and orthogonal (see, for instance, Ref. \cite{keinert:04}). Classical sampling in multiwavelet subspaces has been studied in Refs. \cite{garcia:08,selesnick:99,sun:99}. An example of the formula \eqref{sampling1bis} in the shift-invariant subspace of $L^2(\RR)$ generated by the Hermite  cubic splines can be found in \cite{garcia:08}.

\item In case the operator $T=U$ is unitary, the $L\times L$ (covariance) matrix formed with the cross-correlation $\langle U^ka_m, a_n \rangle _\mathcal{H}$, \,$k\in \mathbb{Z}$\,, for $1\le m,n \le L$,
admits the spectral representation \cite{kolmogorov:41}:
\[
\Big(\langle U^ka_m, a_n \rangle _\mathcal{H} \Big)_{1\le m,n \le L}=\frac{1}{2\pi}\int_{-\pi}^\pi {\rm e}^{{\rm i}k\theta} d\boldsymbol{\mu}_{\mathbf{a}}(\theta)\,, \quad k\in \mathbb{Z}\,.
\]
The spectral measure $\boldsymbol{\mu}_{\mathbf{a}}$ is an $L\times L$ matrix; its entries are the spectral measures associated with the cross-correlation functions $\langle U^ka_m, a_n \rangle _\mathcal{H}$. It can be decomposed into an absolute continuous part and its singular part: Thus we can write $d\boldsymbol{\mu}_{\mathbf{a}}(\theta)=\boldsymbol{\Phi}_{\mathbf{a}}(\theta)d\theta+ d\boldsymbol{\mu}^{s}_{\mathbf{a}}(\theta)$. The following result holds \cite[Theorem 5.1]{hector:14}:
The sequence $\{U^na_l\}_{n\in \mathbb{Z};\,l=1, 2, \dots, L}$ is a Riesz basis for $\mathcal{A}_{\mathbf{a}}$ if and only if the singular part $\boldsymbol{\mu}_{\mathbf{a}}^{s}\equiv 0$ and 
\begin{equation*}\label{condspecmeasvariosgen}
0< \einf_{\theta \in (-\pi, \pi)}\lambda_{\min}\big[{\boldsymbol{\Phi}_{\mathbf{a}}(\theta)}\big]\leq \esup_{\theta \in (-\pi, \pi)}\lambda_{\max}\big[{\boldsymbol{\Phi}_{\mathbf{a}}(\theta)}\big]<\infty \,.
\end{equation*}

In particular, whenever $U$ is the shift operator $f(u) \rightarrow f(u-1)$ in $L^2(\RR)$, the above condition is nothing but a condition involving  the {\em Gramian} matrix-function $G_{\mathbf{\Phi}}(w)$ of the generators $\Phi:=(\varphi_1,\varphi_2,\ldots,\varphi_L)^\top$ defined as
\[ 
G_{\mathbf{\Phi}}(w):=\sum_{n\in\ZZ} \widehat{\Phi}(w+n)\overline{\widehat{\Phi}(w+n)}^\top\,.
\] 
Namely, the sequence $\{\varphi_k(\cdot -n)\}_{n \in \ZZ,k=1,2\ldots,L}$ is a Riesz basis for $\Ac_{\mathbf{\Phi}}$ if and only if there exist two positive constants $m$ and $M$ such that 
$m\mathbb{I}_L  \le  G_{\mathbf{\Phi}}(w) \le M\mathbb{I}_L$ a.e. in $(0,1)$ (see, for instance, Ref.\cite{aldroubi:97}).

\item In the {\em overcomplete} setting we have that $s>rL$. Whenever $s=rL$ we are in the Riesz basis setting: There exist $s$ unique elements $c_j\in \Ac_{\mathbf{a}}$, $j=1, 2, \dots, s$, such that the sequence $\big\{T^{rm}c_j\big\}_{m\in \ZZ;\, j=1,2,\ldots, s}$ is a {\em Riesz basis} for $\mathcal{A}_{\mathbf{a}}$, and the sampling expansion \eqref{sampling1bis} holds. Moreover, due to the uniqueness of the coefficients following a Riesz basis expansion, the {\em interpolation property} 
$\mathcal{L}_{j'} c_j(rm)=\delta_{j,j'}\, \delta_{n,0}$, where $m\in \ZZ$ and $j, j'=1,2,\dots, s$, holds. 
\end{enumerate}

\section{The  cyclic case with a single generator}
\label{section4}
For a fixed $a\in \mathcal{H}$, assume that there exists a nonnegative integer $N$ such that $T^Na=a$; let $N$ be the smallest index with this property. Next, we consider the finite dimensional subspace $\mathcal{A}_a:= \espan \big\{a, Ta, T^2a,\dots,T^{N-1}a\big\}$ in $\Hc$. Assuming that this set of vectors is linearly independent in $\Hc$ we have the $N$-dimensional subspace of $\Hc$
\[
\mathcal{A}_a=\Big\{\sum_{k=0}^{N-1} \alpha(k)\,T^ka \,:\, (\alpha(0), \alpha(1), \dots, \alpha(N-1))^\top\in \CC^N \Big\}\,,
\]
and the isomorphism $\mathcal{T}_{N,a}$ between $\CC^N$ and $\Ac_a$
\[
\begin{array}[c]{ccll}
\mathcal{T}_{N,a}: & \CC^N & \longrightarrow & \mathcal{A}_a\\
       & \boldsymbol{\alpha}=\displaystyle{\sum_{k=0}^{N-1} \alpha(k)\,\mathbf{e_k}} & \longmapsto & x=\displaystyle{\sum_{k=0}^{N-1} \alpha(k)\,T^ka}\,,
\end{array}
\]
where $\big\{\mathbf{e_0}, \mathbf{e_1}, \dots, \mathbf{e_{N-1}} \big\}$ denotes the canonical basis for $\CC^N$.

Let $\{\alpha(k)\}_{k\in \ZZ}$ be an $N$-periodic sequence in $\CC$. For $1\le m\le N-1$ consider the vectors in $\CC^N$
\[
\begin{split}
\boldsymbol{\alpha}_0&:=\big(\alpha(0), \alpha(1),\dots, \alpha(N-1) \big)^\top\quad\text{and}\\
 \boldsymbol{\alpha}_{N-m}&:=\big(\alpha(N-m), \alpha(N-m+1),\dots, \alpha(N-m+N-1) \big)^\top\,.
 \end{split}
\] 
Then, the following {\em $T$-shifting property} holds (its proof is analogous to that in \cite[Proposition 2]{garcia:15})
\begin{equation}
\label{shifting3}
\mathcal{T}_{N,a}(\boldsymbol{\alpha}_{N-m})= T^m \big(\mathcal{T}_{N,a}(\boldsymbol{\alpha}_0)\big)\quad \text{for any\,\, $1\le m\le N-1$}\,.
\end{equation}

\subsubsection*{An expression for the samples}
Let $r$ be a positive integer such that $r|N$, and define $\ell:=N/r$. Fixed $s$ elements $b_j\in \Hc$, $j=1,2,\dots,s$, for each $x\in \Ac_a$ we consider its generalized samples  defined by
\begin{equation}
\label{samples3}
\Lc_jx(rn):=\langle x, (T^*)^{-rn}b_j \rangle_{\Hc}\,, \quad \text{$n=0,1,\ldots,\ell-1$ and $j=1,2,\ldots, s$}\,.
\end{equation}
For $x=\sum_{k=0}^{N-1} \alpha(k)\,T^ka$ we obtain a more convenient expression for its samples
\begin{equation}
\label{samples4}
\begin{split}
\Lc_jx(rn) &= \big\langle \sum_{k=0}^{N-1} \alpha(k)\, T^ka, (T^*)^{-rn}b_j \big\rangle_{\Hc} 
=\sum_{k=0}^{N-1} \alpha(k)\,\langle T^ka, (T^*)^{-rn}b_j \rangle_{\Hc} \\
&=\Big\langle \sum_{k=0}^{N-1} \alpha(k) \,\mathbf{e_k}, \sum_{k=0}^{N-1} \overline{\langle T^{k-rn}a, b_j \rangle}_{\Hc}\,\mathbf{e_k} \Big\rangle_{\CC^N}= \big\langle \boldsymbol{\alpha}, \mathbf{g}_{j,n} \big\rangle_{\CC^N}\,,
\end{split}
\end{equation}
where $\mathbf{g}_{j,n}=\sum_{k=0}^{N-1} \overline{\langle T^{k-rn}a, b_j \rangle}_{\Hc}\,\mathbf{e_k}$. In terms of  the $N$-periodic sequence in $\CC$ defined as
\begin{equation}
\label{r}
r_{a,b_j}(k):=\langle T^ka, b_j\rangle_\Hc\,, \,\,k\in \ZZ\,, 
\end{equation}
we can write
\begin{equation}
\label{g}
\mathbf{g}_{j,n}=\sum_{k=0}^{N-1} \overline{\langle T^{N+k-rn}a, b_j \rangle}_{\Hc}\,\mathbf{e_k}
=\sum_{k=0}^{N-1} \overline{r_{a,b_j}(N+k-rn)}\,\mathbf{e_k}\,.
\end{equation}
Having in mind the expression \eqref{samples4} for the samples $\big\{\Lc_jx(rn)\big\}_{\substack{ j=1,2,\ldots,s\\n=0,1,\ldots,\ell-1 }}$, and the isomorphism 
$\mathcal{T}_{N,a}$, any $x\in\Ac_a$ can be recovered from its samples if and only if the set of vectors 
$\big\{ \mathbf{g}_{j,n}\big\}_{\substack{ j=1,2,\ldots,s\\n=0,1,\ldots,\ell-1 }}$ in $\CC^N$ forms a spanning set for $\CC^N$, i.e., a frame for $\CC^N$ (see, for instance, Refs. \cite{casazza:14,ole:16}). This is equivalent to the condition $\text{rank\,}\mathbb{G}_{a,\mathbf{b}}=N$, where $\mathbb{G}_{a,\mathbf{b}}$ denotes the $N\times s\ell$ matrix whose columns are precisely the vectors $\big\{ \mathbf{g}_{j,n}\big\}_{\substack{ j=1,2,\ldots,s\\n=0,1,\ldots,\ell-1 }}$ written as
\[
\mathbb{G}_{a,\mathbf{b}}:=\begin{pmatrix}
\mathbf{g}_{1,0} & \hdots & \mathbf{g}_{1,\ell-1} & \mathbf{g}_{2,0} & \hdots & \mathbf{g}_{2,\ell-1} & \hdots & \mathbf{g}_{s,0}&\hdots &\mathbf{g}_{s,\ell-1}\\
\end{pmatrix}.
\]
In particular, we have that $N\le s\ell$, that is, $s\geq r$. Having in mind \eqref{g}, $N=r\ell$ and the $N$-periodic character of $r_{a,b_j}(k)$ we obtain that 
\begin{equation}
\label{matrixcc}
\mathbb{G}_{a,\mathbf{b}}=\begin{pmatrix}
\mathbb{R}_{a,b_1}^* & \mathbb{R}_{a,b_2}^* & \hdots & \mathbb{R}_{a,b_s}^* \\
\end{pmatrix}:=\mathbb{R}_{a,\mathbf{b}}^*\,,
\end{equation}
where each $\ell \times N$ block $\mathbb{R}_{a,b_j}$, $j=1,2,\dots,s$\,, is given by
\[
\mathbb{R}_{a,b_j}=\begin{pmatrix}
r_{a,b_j}(0) & r_{a,b_j}(1) & \hdots & r_{a,b_j}(N-1) \\
      r_{a,b_j}(N-r) & r_{a,b_j}(N-r+1)& \hdots & r_{a,b_j}(2N-r-1) \\
           \vdots & \vdots& \ddots & \vdots \\
                r_{a,b_j}(r) & r_{a,b_j}(r+1)& \hdots & r_{a,b_j}(r+N-1) \\
\end{pmatrix}.
\]
For $j=1,2,\dots,s$,  we have the following expression for the samples $\big\{\Lc_j x(rn)\big\}_{n=0}^{\ell-1}$ of $x=\sum_{k=0}^{N-1} \alpha(k) \, T^ka \in \Ac_a$
\[
\big(\Lc_j x(0), \Lc_j x(r), \cdots,  \Lc_j x(r(\ell-1)) \big)^\top=\mathbb{R}_{a,b_j}\, \big(\alpha(0), \alpha(1), \ldots, \alpha(N-1) \big)^\top.
\]
In other words, denoting the vectors $\boldsymbol{\alpha}:=\big(\alpha(0), \alpha(1), \dots, \alpha(N-1) \big)^\top \in \CC^N$ and
\begin{equation*}
\label{L}
\boldsymbol{\Lc}_{\text{sam}}x:=\big(\Lc_1 x(0), \Lc_1 x(r), \dots, \Lc_1 x(r(\ell-1)), \dots, \Lc_s x(0), \dots, \Lc_s x(r(\ell-1)) \big)^\top \in \CC^{s\ell}\,,
\end{equation*}
the matrix relationship  $\boldsymbol{\Lc}_{\text{sam}}x=\mathbb{R}_{a,\mathbf{b}}\,\boldsymbol{\alpha}$ holds where $\mathbb{R}_{a,\mathbf{b}}$ is the $s\ell \times N$ matrix  deduced from  \eqref{matrixcc}.

As $\text{rank\,}\mathbb{R}_{a,\mathbf{b}}=\text{rank\,}\mathbb{G}_{a,\mathbf{b}}=N$, the {\em Moore-Penrose pseudo-inverse} of $\mathbb{R}_{a,\mathbf{b}}$ is the $N\times s\ell$ matrix $\mathbb{R}^\dag_{a,\mathbf{b}}=\big[\mathbb{R}_{a,\mathbf{b}}^*\,\mathbb{R}_{a,\mathbf{b}} \big]^{-1}\mathbb{R}_{a,\mathbf{b}}^*$. Any dual frame of  $\big\{ \mathbf{g}_{j,n}\big\}_{\substack{j=1,2,\ldots,s \\ n=0,1,\ldots,\ell-1}}$ in $\CC^N$ is given by the columns of any left-inverse $\mathbb{H}$ of the matrix $\mathbb{R}_{a,\mathbf{b}}$; i.e., $\mathbb{H}\,\mathbb{R}_{a,\mathbf{b}}=\mathbb{I}_N$. All these matrices are expressed as 
\begin{equation}
\label{left}
\mathbb{H}=\mathbb{R}_{a,\mathbf{b}}^\dag+\mathbb{U}\big[\mathbb{I}_{s\ell}-\mathbb{R}_{a,\mathbf{b}}\,\mathbb{R}_{a,\mathbf{b}}^\dag\big]\,,
\end{equation}
where $\mathbb{U}$ denotes any arbitrary $N\times s\ell$ matrix. Let $\mathbb{H}$ be any {\em left-inverse} of $\mathbb{R}_{a,\mathbf{b}}$, and denote $\mathbf{h}_{j,n}$ its $(j-1)\ell+n+1$ column where $j=1,2,\ldots,s$ and $n=0,1,\ldots,\ell-1$; thus, $\big\{ \mathbf{h}_{j,n}\big\}_{\substack{j=1,2,\ldots,s \\ n=0,1,\ldots,\ell-1}}$ is a dual frame of $\big\{ \mathbf{g}_{j,n}\big\}_{\substack{j=1,2,\ldots,s \\ n=0,1,\ldots,\ell-1}}$.
Given any  $x=\sum_{k=0}^{N-1} \alpha(k) \, T^ka$ in $\Ac_a$,  from the matrix relationship $\boldsymbol{\Lc}_{\text{sam}}x=\mathbb{R}_{a,\mathbf{b}}\,\boldsymbol{\alpha}$ for the corresponding $\boldsymbol{\alpha}=\sum_{k=0}^{N-1} \alpha(k) \,\mathbf{e_k} \in \CC^N$ we obtain
\begin{equation*}
\label{F}
\boldsymbol{\alpha}=\big(\alpha_0, \alpha_1, \dots, \alpha_{N-1} \big)^\top=\mathbb{H}\,\boldsymbol{\Lc}_{\text{sam}}x=\sum_{j=1}^s \sum_{n=0}^{\ell-1}\Lc_jx(rn)\,\mathbf{h}_{j,n}\,.
\end{equation*}
Applying the isomorphism $\mathcal{T}_{N,a}$, for any $x=\sum_{k=0}^{N-1} \alpha(k) \, T^ka \in \Ac_a$ we get
\begin{equation}
\label{sampling2}
x=\mathcal{T}_{N,a}\big(\boldsymbol{\alpha}\big)=\sum_{j=1}^s\sum_{n=0}^{\ell-1}\Lc_jx(rn)\,\mathcal{T}_{N,a}\big(\mathbf{h}_{j,n}\big)\,.
\end{equation}
The column $\mathbf{h}_{j,n}$ in the above formula do not have, in principle,  any suitable structure for applying the $T$-shifting property \eqref{shifting3}. Although we will see that the columns of the Moore-Penrose pseudo-inverse $\mathbb{R}^\dag_{a,\mathbf{b}}$ fulfil the required attribute, we  will construct all the left-inverses of $\mathbb{R}_{a,\mathbf{b}}$ allowing it.
\subsubsection*{A regular sampling formula in $\Ac_a$}
Note that each $\ell\times N$ block $\mathbb{R}_{a,b_j}$ has an $r$-circulant character in the sense that each row of $\mathbb{R}_{a,b_j}$ is the previous row moved to the right $r$ places and wrapped around. In general, and in terms of  a matrix $C$  of order $s\ell\times N$ partitioned into $s$ submatrices of order $\ell\times N$, each block has an $r$-circulant character if and only if 
$ C P_N^r=   \mathbb{P} C
$,  or equivalently, 
$$ C= \mathbb{P}^* C P_N^r $$
where $P_i$ denotes the $1$-circulant square matrix of order $i\in \mathbb{N}$ with first row $(0, 1, 0, \cdots, 0)$ and $ \mathbb{P}$ is the square matrix of order $ s\ell$ given by  $\mathbb{P}= \text{diag}(P_\ell, \cdots, P_\ell)$,  the direct sum of $s$ times the matrix $P_\ell$. The above characterization allows to conclude easily that $(C^\dag)^*$ inherits, and consequently $(C^\dag)^\top$, the $r$-circulant character  
from  $C$. Indeed 
$$ (C^\dag)^*=  ((\mathbb{P}^* C P_N^r)^\dag)^* =\big((P_N^r)^* C^{\dag}\mathbb{P}\big)^*=\mathbb{P}^* (C^\dag)^* P_N^r \, .$$
For more details on pseudoinverses of circulant matrices see Refs.~\cite{pye:73,stallings:72}. In these sources are to be found the  above results although  for a square matrix $C$.

\medskip

We now proceed to construct a specific left-inverse $\mathbb{H}_{\mathbb{S}}$ of $\mathbb{R}_{a,\mathbf{b}}$ from any left-inverse $\mathbb{H}$ given by \eqref{left} in the following way: We denote as $\mathbb{S}$ the first $r$ rows of the matrix $\mathbb{H}$, i.e., $\mathbb{S}\,\mathbb{R}_{a,\mathbf{b}}=\big[\mathbb{I}_r, \mathbb{O}_{r\times (N-r)}\big]$, where $\mathbb{I}_r$ and $\mathbb{O}_{r\times (N-r)}$ denote, respectively, the identity matrix of order $r$ and the zero matrix of order $r\times (N-r)$. According to the structure of the matrix $\mathbb{R}_{a,\mathbf{b}}$ (see \eqref{matrixcc}), we partition the $r\times s\ell$ matrix $\mathbb{S}$ into $\big(\mathbb{S}_1 \,\, \mathbb{S}_2 \,\, \dots \,\,\mathbb{S}_s  \big)$, where each $\mathbb{S}_j$, $j=1,2,\dots,s$, is an $r\times \ell$ block. Now, we form the $N\times s\ell$ matrix $\mathbb{H}_{\mathbb{S}}:=\big( \mathbb{\widetilde{S}}_1\,\, \mathbb{\widetilde{S}}_2\,\, \dots \,\, \mathbb{\widetilde{S}}_s\big)$ by using the columns of 
$\mathbb{S}_j$, $j=1,2,\dots, s$ in the following manner:\\
\noindent $\bullet$ The first column of $\mathbb{\widetilde{S}}_j$ is a concatenation of the columns $1$, $\ell$, $\ell-1$, $\dots$, and  $2$ of $\mathbb{S}_j$; \\
\noindent $\bullet$ The second column of $\mathbb{\widetilde{S}}_j$ is a concatenation of the columns $2$, $1$, $\ell$, $\dots$, and $3$ of $\mathbb{S}_j$; \\
\noindent $\bullet$ The third column of $\mathbb{\widetilde{S}}_j$ is a concatenation of the columns $3$, $2$, $1$, $\dots$, and $4$ of $\mathbb{S}_j$; \\
Repeating the process, finally, \\
\noindent $\bullet$  The column $\ell$ of $\mathbb{\widetilde{S}}_j$ is a concatenation of the columns $\ell$, $\ell-1$, $\ell-2$, $\dots$, and $1$ of $\mathbb{S}_j$. 

The elements of each column of $\mathbb{\widetilde{S}}_j$ are identical to the previous column of $\mathbb{S}_j$ but are moved $r$ positions down with wraparound.

With this procedure, we obtain a left-inverse matrix $\mathbb{H}_{\mathbb{S}}$ for $\mathbb{R}_{a,\mathbf{b}}$, i.e., 
$\mathbb{H}_{\mathbb{S}}\, \mathbb{R}_{a,\mathbf{b}}=\mathbb{I}_N$ (see the proof in \cite[Lemma 2]{garcia:15}). Next we denote the columns of $\mathbb{H}_{\mathbb{S}}$ as
\[
\mathbb{H}_{\mathbb{S}}=\begin{pmatrix}
\mathbf{h}_{1,0} & \hdots & \mathbf{h}_{1,\ell-1} & \mathbf{h}_{2,0} & \hdots & \mathbf{h}_{2,\ell-1} & \hdots & \mathbf{h}_{s,0}&\hdots &\mathbf{h}_{s,\ell-1}\\
\end{pmatrix}
\]
Due to the structure of the columns of $\mathbb{H}_{\mathbb{S}}$, by using the $T$-shifting property \eqref{shifting3}, we have that $\mathcal{T}_{N,a}\big(\mathbf{h}_{j,n}\big)=T^{rn}\mathcal{T}_{N,a}\big(\mathbf{h}_{j,0}\big)$, \,$j=1,2,\ldots,s$ and $n=0,1,\ldots,\ell-1$. As a consequence, formula \eqref{sampling2} reads
\[
x=\sum_{j=1}^s\sum_{n=0}^{\ell-1}\Lc_jx(rn)\,T^{rn} c_{j,h} \,,
\]
where $c_{j,h}=\mathcal{T}_{N,a}\big(\mathbf{h}_{j,0}\big)\in \Ac_a$, $j=1,2, \dots,s$. In fact, the following result holds:

\begin{theorem}
\label{regular2}
Given the $s\ell \times N$ matrix  $\mathbb{R}_{a,\mathbf{b}}$ defined in \eqref{matrixcc}, the following statements are equivalents:
\begin{enumerate}[(a)]
\item $\text{rank\ $\mathbb{R}_{a,\mathbf{b}}$}=N$
\item There exists an $r\times s\ell$ matrix $\mathbb{S}$ such that
\begin{equation*}
\label{S}
\mathbb{S}\,\mathbb{R}_{a,\mathbf{b}}=\big(\mathbb{I}_r, \mathbb{O}_{r\times (N-r)}\big)\,,
\end{equation*}
where $\mathbb{I}_r$ and $\mathbb{O}_{r\times (N-r)}$ denote, respectively, the identity matrix of order $r$ and the zero matrix of order $r\times (N-r)$.
\item There exist $c_j\in \Ac_a$, $j=1,2,\dots, s$ such that the sequence $\big\{T^{rn} c_j\big\}_{\substack{j=1,2,\ldots, s \\ n=0,1,\ldots,\ell-1}}$ is a frame for $\Ac_a$, and for any $x\in \Ac_a$ the expansion
\begin{equation}
\label{sampling2bis}
x=\sum_{j=1}^s \sum_{n=0}^{\ell-1} \Lc_j x(rn)\, T^{rn} c_j
\end{equation}
holds.
\item There exist a frame $\big\{ C_{j,n}\big\}_{\substack{j=1,2,\ldots, s \\ n=0,1,\ldots,\ell-1}}$ for $\Ac_a$ such that, for each $x\in \Ac_a$ the expansion
\[
x=\sum_{j=1}^s\sum_{n=0}^{\ell-1} \Lc_j x(rn)\, C_{j,n}
\]
holds.
\end{enumerate}
\end{theorem}
\begin{proof}
We have already proved that $(a)$ implies $(b)$, and that $(b)$ implies $(c)$. Obviously, $(c)$ implies $(d)$.
The proof concludes as that of Theorem 3.1 in \cite{garcia:15}.
\end{proof}
\subsubsection*{Notes and remarks}
Next we list some specific comments for this section:
\begin{enumerate}
\item The vectors $\big\{T^ka\big\}_{k=0}^{N-1}$ in $\Hc$ are linearly independent if and only if  the $N \times N$ {\em Gram matrix} $\Big( \langle T^la,T^ka\rangle \Big)_{0\leq k, l\leq N-1}$ has non zero determinant.

\item The pseudo-inverse $\mathbb{R}_{a,\mathbf{b}}^\dag$ is computed by using the singular value decomposition of 
$\mathbb{R}_{a,\mathbf{b}}$; the singular values are the square root of the eigenvalues of  the $N\times N$ invertible and positive semidefinite matrix 
$\mathbb{R}_{a,\mathbf{b}}^*\,\mathbb{R}_{a,\mathbf{b}}$ (see, for instance,  \cite{ole:16,horn:99}). Note that the singular value decomposition of $\mathbb{R}_{a,\mathbf{b}}$ is the most reliable method to reveal its rank in practice. 

Besides, the optimal frame bound of the frame $\big\{ \mathbf{g}_{j,n}\big\}_{\substack{j=1,2,\ldots,s \\ n=0,1,\ldots,\ell-1}}$ for $\CC^N$ defined in \eqref{g} are $\sigma_1^2$ and $\sigma_N^2$ where $\sigma_1$ (respectively $\sigma_N$) denotes the smallest (respectively the largest) singular value of the matrix $\mathbb{R}_{a,\mathbf{b}}$.

\item In the {\em overcomplete} setting we have that  $s>r$. Whenever $r=s$, the $N\times N$ matrix $\mathbb{R}_{a,\mathbf{b}}$ is invertible and there exist $r$ unique elements $c_j\in \Ac_a$, $j=1, 2, \dots, r$, such that the sequence $\big\{T^{rn}c_j\big\}_{\substack{j=1,2,\ldots, r \\ n=0,1,\ldots,\ell-1}}$ is a basis for $\mathcal{A}_a$, and the sampling expansion \eqref{sampling2bis} holds. Notice that in this case the inverse matrix $\mathbb{R}_{a,\mathbf{b}}^{-1}$ has necessarily the structure of the matrix $\mathbb{H}_{\mathbb{S}}$. Moreover, due to the uniqueness of the coefficients in a basis expansion, the {\em interpolation property}  $\mathcal{L}_{j'} c_j(rn)=\delta_{j,j'}\, \delta_{n,0}$, where $n=0,1,\ldots,\ell-1$ and $j, j'=1,2,\dots, r$, holds (for a similar result, see \cite[Corollary 4]{garcia:15}).
 
\item The matrix $\mathbb{H}_{\mathbb{S}}$ constructed from the first $r$ rows of a left-inverse $\mathbb{H}$ of $\mathbb{R}_{a,\mathbf{b}}$  belongs to the family described by \eqref{left}; indeed, $\mathbb{H}_{\mathbb{S}}=\mathbb{R}_{a,\mathbf{b}}^\dag+\mathbb{U}_{\mathbb{S}}[\mathbb{I}_{s\ell}-\mathbb{R}_{a,\mathbf{b}}\,\mathbb{R}_{a,\mathbf{b}}^\dag]$ for $\mathbb{U}_{\mathbb{S}}=\mathbb{U}+\mathbb{H}_{\mathbb{S}}-\mathbb{H}$ where 
$\mathbb{U}$ is a matrix such that $\mathbb{H}=\mathbb{R}_{a,\mathbf{b}}^\dag+\mathbb{U}[\mathbb{I}_{s\ell}-\mathbb{R}_{a,\mathbf{b}}\,\mathbb{R}_{a,\mathbf{b}}^\dag]$.

\item An easy example involving  the cyclic shift in the Hilbert space $\ell_N^2(\ZZ)$ of $N$-periodic sequences of complex numbers can be found in \cite{garcia:15}.
\end{enumerate}
\section{The  cyclic case with multiple generators}
\label{section5}
We denote the set of generators as $\mathbf{a}:=\{a_1,a_2,\dots, a_L\}\subset \Hc$ and their respective orders as $N_1, N_2, \dots, N_L$, i.e., $T^{N_l}a_l=a_l$, $l=1,2,\dots,L$. Let $\mathcal{A}_\mathbf{a}$ be the subspace in $\Hc$ defined by
\[
\mathcal{A}_\mathbf{a}:= \espan \big\{a_l, Ta_l, T^2a_l,\dots,T^{N_l-1}a_l\big\}_{l=1}^L \,.
\]
Linear independence of the above vectors  gives an  $N_1+N_2+\dots+N_L$ dimensional subspace in $\Hc$ described as
\[
\Ac_{\mathbf{a}}=\Big\{ \sum_{l=1}^L\sum_{k=0}^{N_l-1}\alpha^l(k)\, T^k a_l \,\,:\,\,\alpha^l(k)\in \CC \Big\}\,.
\]
Next we consider  the isomorphism $\mathcal{T}_{\mathbf{N},\mathbf{a}}$ between $\CC^{N_1+N_2+\dots+N_L}$ and $\Ac_{\mathbf{a}}$
\[
\begin{array}[c]{cll}
\mathcal{T}_{\mathbf{N},\mathbf{a}}:  \CC^{N_1+N_2+\dots+N_L} & \longrightarrow & \mathcal{A}_{\mathbf{a}}\\
        \boldsymbol{\alpha}:=(\boldsymbol{\alpha}^1, \boldsymbol{\alpha}^2,\dots,\boldsymbol{\alpha}^L)^\top & \longmapsto & x=\displaystyle{\sum_{l=1}^L\sum_{k=0}^{N_l-1} \alpha^l(k)\,T^ka_l}\,,
\end{array}
\]
where $\boldsymbol{\alpha}^l:=\big(\alpha^l(k)\big)_{k=0}^{N_l-1}$ for $l=1, 2, \dots, L$.

Consider any vector 
$\boldsymbol{\alpha}_0\in  \CC^{N_1+N_2+\dots+N_L}$ denoted by $\boldsymbol{\alpha}_0=\big(\boldsymbol{\alpha}^1_0, \boldsymbol{\alpha}^2_0, \dots,\boldsymbol{\alpha}^L_0 \big)^\top$, where each block, denoted by $\boldsymbol{\alpha}^l_0=\big( \alpha^l(0), \alpha^l(1),\dots, \alpha^l(N_l-1) \big)$, \,$1\le l \le L$, is a row vector of dimension $N_l$. Set $N:={\rm l.c.m.}\,(N_1, N_2, \dots N_L)$ and $1\le m \le N-1$.  From  $\boldsymbol{\alpha}_0$ we define  a new vector  
$\boldsymbol{\alpha}_{N-m}:=\big(\boldsymbol{\alpha}^1_{N-m}, \boldsymbol{\alpha}^2_{N-m}, \dots,\boldsymbol{\alpha}^L_{N-m} \big)^\top$ in $\CC^{N_1+N_2+\dots+N_L}$, where each block $\boldsymbol{\alpha}^l_{N-m}=\big( \alpha^l(N-m), \alpha^l(N-m+1),\dots, \alpha^l(N-m+N_l-1) \big)$,
is obtained by assuming an $N_l$-periodic character in each $\alpha^l(\cdot)$, \,$l=1, 2, \dots, L$. Then  the following {\em T-shifting property} holds (its proof is analogous to that in \cite[Lemma 1]{garcia:16})
\begin{equation}
\label{shifting4}
\mathcal{T}_{\mathbf{N},\mathbf{a}}(\boldsymbol{\alpha}_{N-m})= T^m \big(\mathcal{T}_{\mathbf{N},\mathbf{a}}(\boldsymbol{\alpha}_0)\big)\quad \text{for any\,\, $1\le m\le N-1$}\,.
\end{equation}

\subsubsection*{An expression for the samples}
We consider a sampling period $r$ which divides $N={\rm l.c.m.}\,(N_1, N_2, \dots N_L)$ and $\ell:=N/r$. Fixed $s$ elements $b_j\in \Hc$, $j=1,2,\dots,s$, for each $x\in \Ac_{\mathbf{a}}$ we consider its $s\ell$ generalized samples $\big\{\Lc_jx(rn)\big\}_{\substack{j=1,2,\ldots, s \\ n=0,1,\ldots,\ell-1}}$ with sampling period $r$  defined by
\[
\Lc_jx(rn):=\langle x, (T^*)^{-rn}b_j \rangle_{\Hc}\,, \quad \text{$n=0,1,\ldots,\ell-1$ and $j=1,2,\ldots, s$}\,.
\]
Note that the samples are $N$-periodic. Proceeding as in the case of a single generator, for each $x=\sum_{l=1}^L\sum_{k=0}^{N_l-1} \alpha^l(k)\,T^ka_l$ we get a similar expression for its samples. Namely,
\begin{equation}
\label{samples5}
\Lc_jx(rn) =\sum_{l=1}^L\big\langle \boldsymbol{\alpha}^l, \mathbf{g}_{j,n}^l \big\rangle_{\CC^{N_l}}=\big\langle \boldsymbol{\alpha}, \mathbf{g}_{j,n} \big\rangle_{\CC^{N_1+N_2+\dots+N_L}}\,,
\end{equation}
where  $\mathbf{g}_{j,n}^l:=\Big(\overline{r_{a_l,b_j}(N+k-rn)}\Big)_{k=0}^{N_l-1}$  and
$\mathbf{g}_{j,n}:=(\mathbf{g}_{j,n}^1, \mathbf{g}_{j,n}^2, \dots, \mathbf{g}_{j,n}^L)^\top$, for $l=1, 2, \dots, L$ and $j=1,2,\dots,s$.

As a consequence of \eqref{samples5}, and having in mind the equivalence between spanning sets and frames in finite dimensional spaces, we obtain that any $\boldsymbol{\alpha} \in \CC^{N_1+N_2+\dots+N_L}$ (and consequently any $x\in \Ac_{\mathbf{a}}$) can be recovered from the sequence of samples if and only if the set of vectors 
$\big\{ \mathbf{g}_{j,n}\big\}_{\substack{ j=1,2,\ldots,s\\n=0,1,\ldots,\ell-1 }}$ forms a spanning set (frame) for $\CC^{N_1+N_2+\dots+N_L}$.

\medskip

As in the single generator case, expression \eqref{samples5} for the samples can be written in matrix form as
\begin{equation}
\label{matrixexp}
\boldsymbol{\Lc}_{\text{sam}}x= \mathbb{R}_{\mathbf{a},\mathbf{b}}\, \boldsymbol{\alpha}\,,
\end{equation}
where $\boldsymbol{\Lc}_{\text{sam}} x=\big(\boldsymbol{\Lc}_1 x, \boldsymbol{\Lc}_2 x, \dots, \boldsymbol{\Lc}_s x \big)^\top \in \CC^{s\ell}$, $\boldsymbol{\alpha}:=(\boldsymbol{\alpha}^1, \boldsymbol{\alpha}^2,\dots,\boldsymbol{\alpha}^L)^\top$ in $\CC^{N_1+N_2+\dots+N_L}$ is the vector associated to $x=\sum_{l=1}^L\sum_{k=0}^{N_l-1} \alpha^l(k)\,T^ka_l\in \Ac_{\mathbf{a}}$, and $\mathbb{R}_{\mathbf{a},\mathbf{b}}$ is the $s\ell \times(N_1+N_2+\dots+N_L)$ matrix of the form
\[
\mathbb{R}_{\mathbf{a},\mathbf{b}}=\Big( \mathbb{R}_{a_l, b_j}\Big)_{\substack{j=1,2,\ldots, s\\ l=1,2,\ldots,L }}
\]
where the block $\mathbb{R}_{a_l, b_j}$ is the $\ell\times N_l$ matrix 
{\small
\[
\mathbb{R}_{a_l,b_j}=
\begin{pmatrix}
r_{a_l,b_j}(0) & r_{a_l,b_j}(1) & \hdots & r_{a_l,b_j}(N_l-1) \\
      r_{a_l,b_j}(N-r) & r_{a_l,b_j}(N-r+1)& \hdots & r_{a_l,b_j}(N-r+N_l-1) \\
           \vdots & \vdots& \ddots & \vdots \\
                r_{a_l,b_j}(N-r(\ell-1)) & r_{a_l,b_j}(N-r(\ell-1)+1)& \hdots & r_{a_l,b_j}(N-r(\ell-1)+N_l-1) \\
\end{pmatrix}\,.
\]
}
In terms of the matrix $\mathbb{R}_{\mathbf{a},\mathbf{b}}$,  any element $x\in \Ac_{\mathbf{a}}$ can be recovered from the sequence of its  generalized samples $\big\{\Lc_jx(rn)\big\}_{\substack{j=1,2,\ldots, s\\ n=0,1,\ldots,\ell-1 }}$ if and only if the matrix $\mathbb{R}_{\mathbf{a},\mathbf{b}}$ has rank  $N_1+N_2+\dots+N_L$.

Following the same procedure that in  Section \ref{section4},  from \eqref{matrixexp} we get
\begin{equation}
\label{H}
\boldsymbol{\alpha}=\mathbb{H}\,\boldsymbol{\Lc}_{\text{sam}}x=\sum_{j=1}^s\sum_{n=0}^{\ell-1}\Lc_jx(rn)\,\mathbf{h}_{j,n}\,,
\end{equation}
where $\mathbf{h}_{j,n}$ denotes the $(j-1)\ell+n+1$ column of any left-inverse $\mathbb{H}$ of the matrix $\mathbb{R}_{\mathbf{a},\mathbf{b}}$. Applying the isomorphism $\mathcal{T}_{\mathbf{N},\mathbf{a}}$ in \eqref{H}, for any 
$x \in \Ac_{\mathbf{a}}$ we obtain
\begin{equation}
\label{noestructura}
x=\mathcal{T}_{\mathbf{N},\mathbf{a}}\big(\boldsymbol{\alpha}\big)=\sum_{j=1}^s\sum_{n=0}^{\ell-1}\Lc_jx(rn)\,\mathcal{T}_{\mathbf{N},\mathbf{a}}\big(\mathbf{h}_{j,n}\big)\,.
\end{equation}
The sampling functions $\mathcal{T}_{\mathbf{N},\mathbf{a}}\big(\mathbf{h}_{j,n}\big)$ in the above formula do not have, in principle, any special $T$-structure. As in Section \ref{section4}, we must find left-inverses of 
$\mathbb{R}_{\mathbf{a},\mathbf{b}}$ 
whose columns have a suitable structure allowing to apply the $T$-shifting property \eqref{shifting4}.
\subsubsection*{A regular sampling formula in $\Ac_\mathbf{a}$}
We proceed to  construct  left-inverses of $\mathbb{R}_{\mathbf{a},\mathbf{b}}$ having the required structure for applying property \eqref{shifting4} to their columns. Indeed, let $\mathbb{H}$ be any left-inverse of the matrix $\mathbb{R}_{\mathbf{a},\mathbf{b}}$. This  matrix $\mathbb{H}$ has order $(N_1+N_2+\dots+N_L)\times s\ell$ and it can be written in blocks as
\begin{equation*}
\mathbb{H}=
\begin{pmatrix}
\mathbb{S}_1^1&\mathbb{S}_2^1&\dots&\mathbb{S}_s^1 \\
\mathbb{S}_1^2&\mathbb{S}_2^2&\dots&\mathbb{S}_s^2 \\ 
\vdots &\vdots&\ddots&\vdots \\ 
\mathbb{S}_1^L&\mathbb{S}_2^L&\dots&\mathbb{S}_s^L 
\end{pmatrix}\,,
\end{equation*}
where each block $\mathbb{S}_j^l$ denotes an $N_l\times \ell$ matrix, $l=1,2,\dots,L$ and $j=1,2,\dots,s$. The idea consists in defining  a new left-inverse 
$\mathbb{\widetilde{H}}$ of $\mathbb{R}_{\mathbf{a},\mathbf{b}}$ to be the modification to $\mathbb{H}$  by replacing every block $\mathbb{S}_j^l$  such that $N_l > r$  by a suitable  $N_l\times \ell$ block $\mathbb{\widetilde{S}}_j^l$ proceeding, for $N$, as in Section \ref{section4}, and then taking the first $N_l$ rows.  In case $N_l \leq r$, we take $\mathbb{\widetilde{S}}_j^l := \mathbb{S}_j^l$. Finally, consider the new $(N_1+N_2+\dots+N_L)\times s\ell$ matrix $\mathbb{\widetilde{H}}$ written in blocks as
\begin{equation*}
\label{Htilde}
\mathbb{\widetilde{H}}=
\begin{pmatrix}
\mathbb{\widetilde{S}}_1^1&\mathbb{\widetilde{S}}_2^1&\dots&\mathbb{\widetilde{S}}_s^1 \\
\mathbb{\widetilde{S}}_1^2&\mathbb{\widetilde{S}}_2^2&\dots&\mathbb{\widetilde{S}}_s^2 \\ 
\vdots &\vdots&\ddots&\vdots \\ 
\mathbb{\widetilde{S}}_1^L&\mathbb{\widetilde{S}}_2^L&\dots&\mathbb{\widetilde{S}}_s^L 
\end{pmatrix}
\end{equation*}
With the above procedure we have obtained a left-inverse matrix $\mathbb{\widetilde{H}}$ for $\mathbb{R}_{\mathbf{a},\mathbf{b}}$ having the desired structure (see \cite[Lemma 2]{garcia:16} for a proof), i.e., for $j=1,2, \dots, s$ and 
$n=1, 2,\dots,\ell-1$ the column $\mathbf{\widetilde{h}}_{j,n}$ of $\mathbb{\widetilde{H}}$ is obtained from the column 
$\mathbf{\widetilde{h}}_{j,0}$ as follows: If we denote the column $\mathbf{\widetilde{h}}_{j,0}$ by 
\[
\mathbf{\widetilde{h}}_{j,0}=\big(\mathbf{\tilde{s}}^1_{j,0}\,, \mathbf{\tilde{s}}^2_{j,0}\,, \dots, \mathbf{\tilde{s}}^L_{j,0}\,\big)^\top \in \CC^{N_1+N_2+\dots+N_L}
\]
where, for each $l=1,2,\dots, L$, $\mathbf{\tilde{s}}^l_{j,0}$ denotes the row vector (of dimension $N_l$)
\[
\mathbf{\tilde{s}}^l_{j,0}=\big(\tilde{s}^l_j(0), \tilde{s}^l_j(1), \dots, \tilde{s}^l_j(N_l-1) \big)\,,
\]
then 
\[
\mathbf{\widetilde{h}}_{j,n}=\big(\mathbf{\tilde{s}}^1_{j,n}\,, \mathbf{\tilde{s}}^2_{j,n}\,, \dots, \mathbf{\tilde{s}}^L_{j,n}\,\big)^\top \in \CC^{N_1+N_2+\dots+N_L}
\]
where, for each $l=1,2,\dots, L$, $\mathbf{\tilde{s}}^l_{j,n}$ denotes the row vector (of dimension $N_l$)
\[
\mathbf{\tilde{s}}^l_{j,n}=\big(\tilde{s}^l_j(N-rn), \tilde{s}^l_j(N-rn+1), \dots, \tilde{s}^l_j(N-rn+N_l-1) \big)\,
\]
(we are also extending each $\tilde{s}^l_j(\cdot)$ taking into account $N_l$-periodicity). 

\medskip

\noindent Thus, property \eqref{shifting4} gives $\mathcal{T}_{\mathbf{N},\mathbf{a}}\big(\mathbf{\widetilde{h}}_{j,n}\big)=T^{rn} \big(\mathcal{T}_{\mathbf{N},\mathbf{a}}\big(\mathbf{\widetilde{h}}_{j,0} \big)\big)$ for $n=0,1,\dots, \ell-1$ and $j=1, 2,\dots,s$.  Consequently, for each $x\in \Ac_{\mathbf{a}}$ formula \eqref{noestructura} reads
\begin{equation}
\label{sampling2tri}
x=\sum_{j=1}^s\sum_{n=0}^{\ell-1}\Lc_jx(rn)\,T^{rn} \big(\mathcal{T}_{\mathbf{N},\mathbf{a}}\big(\mathbf{\widetilde{h}}_{j,0} \big)\big)=
\sum_{j=1}^s\sum_{n=0}^{\ell-1}\Lc_jx(rn)\,T^{rn} c_{j,h}\,,
\end{equation}
where $c_{j,h}=\mathcal{T}_{\mathbf{N},\mathbf{a}}\big(\mathbf{\widetilde{h}}_{j,0} \big)\in \Ac_{\mathbf{a}}$, $j=1,2,\dots,s$. A characterization of the sampling formulas \eqref{sampling2tri}, similar to that in Theorem \ref{regular2}, can be stated for this multiple generators setting (for a comparable result, see \cite[Theorem 3]{garcia:16}).
\subsubsection*{A filter-bank interpretation}
The obtained sampling formulas \eqref{sampling1}, \eqref{sampling1bis}, \eqref{sampling2} and \eqref{sampling2tri} can be implemented as {\em filter-banks}. Let us show it, for instance, in the case included in this section.
Assume that rank of $\mathbb{R}_{\mathbf{a},\mathbf{b}}$ equals $N_1+N_2+\dots+N_L$, and let $\mathbb{\widetilde{H}}$ be a structured left-inverse  of $\mathbb{R}_{\mathbf{a},\mathbf{b}}$ with columns $\mathbf{\widetilde{h}}_{j,n}$, $j=1,2,\dots, s$ and $n=0, 1, \dots , \ell-1$. In the corresponding sampling formula \eqref{sampling2tri} we have 
$c_{j,h}=\mathcal{T}_{N,a}\big(\mathbf{\widetilde{h}}_{j,0}\big)$, $j=1,2,\dots, s$;  denote the components of $\mathbf{\widetilde{h}}_{j,0}$ as the $N_1+N_2+\dots+N_L$ dimensional vector
\[
\mathbf{\widetilde{h}}_{j,0}=\big(\beta_j^1(0), \beta_j^1(1), \dots, \beta_j^1(N_1-1), \dots, \beta_j^L(0), \beta_j^L(1), \dots, \beta_j^L(N_L-1) \big)^\top\,.
\]
Substituting in \eqref{sampling2tri}, for $x\in \Ac_{\mathbf{a}}$ we get
\[
\begin{split}
x&=\sum_{j=1}^s \sum_{n=0}^{\ell-1} \Lc_j x(rn)\, T^{rn} \Big(\sum_{l=1}^L\sum_{k=0}^{N_l-1}\beta_j^l(k)\,T^ka_l\Big) \\
&=\sum_{j=1}^s \sum_{n=0}^{\ell-1} \Lc_j x(rn)\, \Big(\sum_{l=1}^L\sum_{k=0}^{N_l-1}\beta_j^l(m)\,T^{rn+k}a_l\Big).
\end{split}
\]
The change of index $m:=rn+k$ and assuming an $(N_1, N_2, \dots,N_L)$-periodic character of each $\mathbf{\widetilde{h}}_{j,0}$, $j=1,2,\dots, s$, gives
\[
\begin{split}
x &=\sum_{j=1}^s \sum_{n=0}^{\ell-1} \Lc_j x(rn)\, \Big(\sum_{l=1}^L\sum_{m=rk}^{rk+N_l-1}\beta_j^l(m-rn)\,T^{m}a_l\Big) \\
&=\sum_{j=1}^s \sum_{n=0}^{\ell-1} \Lc_j x(rn)\, \Big(\sum_{l=1}^L\sum_{m=0}^{N_l-1}\beta_j^i(m-rn)\,T^{m}a_l\Big) \\
&=\sum_{l=1}^L\sum_{m=0}^{N_l-1}\Big\{\sum_{j=1}^s \sum_{n=0}^{\ell-1}\Lc_j x(rn)\,\beta_j^l(m-rn)\Big\}T^ma_l\,.
\end{split}
\]
In other words, for each $x=\sum_{l=1}^L\sum_{m=0}^{N_l-1} \alpha^l(m)\, T^ma_l$ in $\Ac_{\mathbf{a}}$, the coefficients $\alpha^l(m)$, $m=0,1, \dots, N_l-1$ are, for each $l=1, 2,\dots, L$, the output of a filter-bank
\[
\alpha^l(m)=\sum_{j=1}^s \sum_{n=0}^{\ell-1}\Lc_j x(rn)\,\beta_j^l(m-rn)\,, \quad m=0,1, \dots, N_l-1\,,
\]
involving the input data $\big\{\Lc_j x(rn)\big\}_{\substack{j=1,2,\ldots, s \\ n=0,1,\ldots,\ell-1}}$ and the columns $\mathbf{\widetilde{h}}_{j,0}$, $j=1,2,\dots, s$, of 
the matrix $\mathbb{\widetilde{H}}$ as the impulse responses. 
\subsubsection*{Notes and remarks}
Some specific comments are in order:
\begin{enumerate}
\item  Multiply generated $\mathcal{A}_\mathbf{a}$ spaces involving the shift operator  are very natural in signal theory. There are examples where a single generator fails to describe the appropriate signal subspace: for instance to describe subspaces of periodic extensions of finite signals, several generators  are required (see \cite[Section IV]{garcia:16}).
\item The {\em Gram matrix} of the vectors $\big\{a_l, Ta_l, T^2a_l,\dots,T^{N_l-1}a_l\big\}_{l=1}^L$, which involves the $N_j \times N_i$ matrices  $\mathbf{C}_{a_i, a_j}:=\Big( \langle T^la_i,T^ka_j\rangle \Big)_{0\leq k\leq N_j-1;\;0\leq l\leq N_i-1}$, where $1\le i, j \le L$, gives a necessary and sufficient condition for their linear independence. Namely, they are linearly independent if and only if
\[
\det \begin{pmatrix}\mathbf{C}_{a_1, a_1}&\mathbf{C}_{a_2, a_1}&\dots&\mathbf{C}_{a_L, a_1} \\\mathbf{C}_{a_1, a_2}&\mathbf{C}_{a_2, a_2}&\dots&\mathbf{C}_{a_L, a_2} \\ \vdots &\vdots&\ddots&\vdots \\ \mathbf{C}_{a_1, a_L}&\mathbf{C}_{a_2, a_L}&\dots&\mathbf{C}_{a_L, a_L} \end{pmatrix}\neq 0\,.
\] 
\item In the {\em overcomplete} setting we have that $s\ell>N_1+N_2+\dots+N_L$. Whenever $s\ell=N_1+N_2+\dots+N_L$ we are in the basis setting whenever the matrix $\mathbb{R}_{\mathbf{a},\mathbf{b}}$ is invertible: There exist $s$ unique elements $c_j\in \Ac_\mathbf{a}$, $j=1, 2, \dots, s$, such that the sequence $\big\{T^{rn}c_j\big\}_{\substack{j=1,2,\ldots, s \\ n=0,1,\ldots,\ell-1}}$ is a basis for $\mathcal{A}_\mathbf{a}$, and the sampling expansion \eqref{sampling2tri} holds. Notice that in this case the inverse matrix $\mathbb{R}_{\mathbf{a},\mathbf{b}}^{-1}$ has necessarily the structure of the matrix $\mathbb{\widetilde{H}}$. Moreover, due to the uniqueness of the coefficients in a basis expansion, the {\em interpolation property} 
$\mathcal{L}_{j'} c_j(rn)=\delta_{j,j'}\, \delta_{n,0}$, where $n=0,1,\ldots,\ell-1$ and $j, j'=1,2,\dots, s$, holds (for a similar result, see \cite[Corollary 4]{garcia:16}).
\item Some easy examples involving the whole-point symmetry or the half-point symmetry extensions of  finite signals (see, for instance, \cite{strang:96}) can be found in \cite{garcia:16}. 
\end{enumerate}
\section{The general case associated with an LCA group $G$}
\label{section6}
Let $(G,+)$ be a second countable {\em locally compact abelian} (LCA) Hausdorff group. Let $M< H< G$ be countable (finite or countably infinite) {\em uniform lattices} in $G$. Recall that a  uniform lattice $K$ in $G$ is a discrete subgroup of $G$ such that the quotient group $G/K$ is compact (see, for instance, Ref.~\cite{cabrelli:10}). 
It is known that if $M< H$ are uniform lattices in $G$ then $H/M$ is a finite group (see \cite[Remark 2.2]{cabrelli:12}).

The dual group of the subgroup $H< G$, that is, the set of continuous characters on $H$ is denoted by $\widehat{H}$. Since $H$ is discrete,  its dual $\widehat{H}$ is compact. We assume that its Haar measure $m_{\widehat{H}}$ is normalized to $m_{\widehat{H}}(\widehat{H})=1$. The value of the character $\gamma\in\widehat{H}$ at the point $h\in H$ is denoted by $(h,\gamma)\in\mathbb{T}$.  With this Haar measure normalization the sequence $\{\chi_h\}_{h\in H}$ defined by
\[
\widehat{H}\ni \gamma \mapsto \chi_h(\gamma)=(h,\gamma)\in \mathbb{T}
\]
turns out to be an orthonormal basis for $L^2(\widehat{H})$ (see, for instance, \cite[Prop.\,4.3]{folland:95}). 

\medskip

Let $g\in G\mapsto \Pi(g)$ a {\em group representation} of $G$ on a complex separable Hilbert space $\mathcal{H}$; i.e., $\Pi$ is a mapping from $G$ into the space of bounded invertible operators on $\Hc$, satisfying that $\Pi(g+g')=\Pi(g)\Pi(g')$ for all $g, g'\in G$.

Therefore, the mapping $h\in H\mapsto \Pi(h)$ is a group representation of $H$ on $\mathcal{H}$. For a fixed $a\in\mathcal{H}$ let define the subspace in 
$\mathcal{H}$
\[
\mathcal{A}_a:=\overline{\operatorname{span}}\big\{\Pi(h)a\,:\, h\in H\big\}\subset \mathcal{H}\,.
\]
We assume that $\{\Pi(h)a\}_{h\in H}$ is a Riesz sequence in $\mathcal{H}$. Thus, the subspace $\mathcal{A}_a$ can be expressed as
\[
 \mathcal{A}_a=\Big\{  \sum_{h\in H}\alpha_h \Pi(h)a\, :\, \{\alpha_h\}_{h\in H}\in \ell^2(H)\Big\}\subset \mathcal{H}\,.
\]
As usual, $\{\alpha_h\}_{h\in H}\in \ell^2(H)$ means that $\sum_{h\in H} |\alpha_h|^2 <\infty$. The subspace 
$\mathcal{A}_a$ is the image of the Hilbert space $L^2(\widehat{H})$ by means of the isomorphism:
\[
\begin{array}[c]{ccll}
\mathcal{T}_{H,a}: & L^2(\widehat{H}) & \longrightarrow & \mathcal{A}_a\\
       & \displaystyle{F=\sum_{h\in H} \alpha_h \chi_h} & \longmapsto & \displaystyle{x=\sum_{h\in H} \alpha_h \Pi(h)a}
\end{array}
\]
This isomorphism $\mathcal{T}_{H,a}$ has the following {\em $\Pi$-shifting property} (its proof is analogous to that in \cite[Proposition 1]{garcia:17}):
\begin{equation}
\label{shifting5}
 \mathcal{T}_{H,a}(F\chi_k)=\Pi(k)(\mathcal{T}_{H,a} F)\quad \text{ for any $F\in L^2(\widehat{H})$ and $k\in H$}\,.
\end{equation}

\subsubsection*{An expression for the samples}
Suppose that $s$ vectors $b_j\in \Hc$, $j=1,2,\dots,s$, are given. For each $x\in \mathcal{A}_a$ we define the sequence of its samples taken at the subgroup $M$, as
\begin{equation}
\label{samples6}
   \mathcal{L}_jx(m)=\langle x, \Pi^*(-m)b_j\rangle_\mathcal{H},\, \text{ $m\in M$ and $j=1,2,\dots,s$}\,,
\end{equation}
where $\Pi^*(-m)$ denotes the adjoint operator of $\Pi(-m)$.
For each $x\in \mathcal{A}_a$, let $F$ be  the element in $L^2(\widehat{H})$ such that
$\mathcal{T}_{H,a}F=x$. An alternative expression for the sample $\mathcal{L}_jx(m)$,\, $j=1,2,\dots,s$ and $m\in M$ is
\[
 \mathcal{L}_jx(m)=\Big\langle \sum_{h\in H} \alpha_h \Pi(h)a, \Pi^*(-m)b_j\Big\rangle_\mathcal{H}
  =\sum_{h\in H} \alpha_h \overline{\big\langle \Pi^*(h-m)b_j, a\big\rangle}_\mathcal{H}\,.
\]
Therefore, for any fixed $m\in M$ we have
\[
\mathcal{L}_jx(m)=\Big\langle F, \sum_{h\in H} \langle \Pi^*(h-m)b_j, a\rangle_\mathcal{H}\, \chi_h\Big\rangle_{L^2(\widehat{H})}=
\Big\langle F,\Big(\sum_{k\in H} \langle \Pi^*(k)b_j, a\rangle_\mathcal{H} \,\chi_{k}\Big)\chi_m\Big\rangle_{L^2(\widehat{H})}\,,
\]
where $k=h-m$ runs over $H$. Hence, we obtain the expression
\begin{equation}
\label{samples7}
\mathcal{L}_jx(m)=\big\langle F,\overline{G}_j\,\chi_m\big\rangle_{L^2(\widehat{H})},\, \text{ $m\in M$ and $j=1,2,\dots,s$}\,,
\end{equation}
where the function $G_j\in L^2(\widehat{H})$ is given by $G_j=\sum_{k\in H} \mathcal{L}_ja(k)\,\chi_{-k}$, $j=1,2,\dots,s$.
As a consequence of expression \eqref{samples7}, the recovery of any $x\in\mathcal{A}_a$ depends on the frame property of the sequence
$\big\{\overline{G}_j\,\chi_m\big\}_{m\in M;\, j=1,2,\dots,s}$ in $L^2(\widehat{H})$. This study has been done in \cite[Proposition 2]{garcia:17}. 
In order to state the result we need to introduce some necessary preliminaries. The {\em annihilator} of $M$ in $\widehat{H}$ is the closed subgroup
\[
  M^\bot=\{\gamma\in \widehat{H} \,:\, (m,\gamma)=1 \text{ for all $m$ in $M$ }\}
\]
Since $M^\bot$ is isomorphic to $\widehat{H/M}$, and $H/M$ is finite, the annihilator $M^\bot$ is a finite subgroup of 
$\widehat{H}$. Let $r$ be the order of $M^\perp$ and set $M^\bot=\big\{\mu_0^\bot=0,\mu_1^\bot,\dots,\mu_{r-1}^\bot\big\}$.
It is known that there exists a measurable (Borel) {\em section $\Omega$} of $\widehat{H}/M^\bot$ (see the seminal Ref.~\cite{feldman:68}), i.e.,  a measurable set  $\Omega$ such that
\[
\widehat{H}= \bigcup_{n=0}^{r-1} (\mu^\bot_{n}+\Omega)\quad \text{and}\quad(\mu^\bot_{n}+\Omega) \, \cap \,(\mu^\bot_{n'}+\Omega)=\emptyset, \quad \text{for}\, \, n\neq n'\,.
\]
Notice that $m_{\widehat{H}}(\Omega)=1/r$.  Besides, the sequence $\{\chi_m\}_{m\in M}$ is an orthogonal basis for $L^2(\Omega)$.

\medskip

For $G_{j}\in L^2(\widehat{H})$, $j=1,2,\ldots,s$, we consider the associated $s\times r$ matrix given by
\begin{equation}
\label{Gmatrix}
\mathbb{G}(\xi):=\bigg(G_j\Big(\xi+\mu_k^\bot\Big)\bigg)_{\substack{j=1,2,\ldots,s \\ k=0,1,\ldots, r-1}}\,, \qquad \xi \in \Omega\,,
\end{equation}
and the related constants
\[
\alpha_{\mathbb{G}}:=\einf_{\xi \in \Omega}\lambda_{\min}[\mathbb{G}^*(\xi)\mathbb{G}(\xi)]\,; \quad
\beta_{\mathbb{G}}:=\esup_{\xi \in \Omega}\lambda_{\max}[\mathbb{G}^*(\xi)\mathbb{G}(\xi)]\,.
\]
Thus we have (\cite[Proposition 2]{garcia:17}):

{\em The sequence $\{\overline{G}_j\,\chi_m\}_{m\in M;\,j=1,2,\ldots,s}$  is a frame for $L^2(\widehat{H})$ if and only if $0<\alpha_{\mathbb{G}}\le
\beta_{\mathbb{G}}<\infty$. In this case, the optimal frame bounds are $\alpha_{\mathbb{G}}/r$ and  $\beta_{\mathbb{G}}/r$}.

We also need to characterize its dual frames having the same structure. This is done in \cite[Proposition 4]{garcia:17}: 

{\em Assume that the functions $H_j\in L^\infty(\widehat{H})$, $j=1, 2, \ldots,s$, satisfy
\[
\big(H_1(\xi), H_2(\xi), \ldots, H_s(\xi)\big)\,\mathbb{G}(\xi)=(1, 0, \dots, 0)\,,\quad \text{a.e.}\,\,\, \xi\in \widehat{H}\,.
\]
Then, the sequences  $\{\overline{G}_j\,\chi_m\}_{m\in M;\,j=1,2,\ldots,s}$  and $\{rH_j\,\chi_m\}_{m\in M;\,j=1,2,\ldots,s}$ form a pair of dual frames for $L^2(\widehat{H})$.}

\medskip

All the possible vectors $\big(H_1(\xi), H_2(\xi), \ldots, H_s(\xi)\big)$ satisfying the above condition, and with entries in $L^\infty(\widehat{H})$ are given by the first row of the $r\times s$ matrices
\[
\mathbb{H}(\xi):=\mathbb{G}^\dag(\xi)+\mathbb{U}(\xi)\big[\mathbb{I}_s-\mathbb{G}(\xi)\mathbb{G}^\dag(\xi)\big]\,
\]
where $\mathbb{G}^\dag(\xi)=\big[\mathbb{G}^*(\xi)\,\mathbb{G}(\xi)\big]^{-1}\,\mathbb{G}^*(\xi)$ denotes the {\em Moore-Penrose pseudo-inverse} of 
$\mathbb{G}(\xi)$, and $\mathbb{U}(\xi)$ denotes any $r\times s$ matrix with entries in $L^\infty(\widehat{H})$.
\subsubsection*{A regular sampling formula in $\Ac_a$}
For $x\in\Ac_a$ let $F\in L^2(\widehat{H})$ be such that $\mathcal{T}_{H,a} F=x$. Expanding $F$ with respect to the pair of dual frames $\{\overline{G}_j\,\chi_m\}_{m\in M;\,j=1,2,\ldots,s}$  and $\{rH_j\,\chi_m\}_{m\in M;\,j=1,2,\ldots,s}$, and having in mind \eqref{samples7} we obtain
\[
 F=\sum_{j=1}^s\sum_{m\in M} \big\langle F,\overline{G}_j\,\chi_m\big\rangle_{L^2(\widehat{H})} \, rH_j\,\chi_m=\sum_{j=1}^s\sum_{m\in M} \mathcal{L}_jx(m) \, rH_j\,\chi_m \quad \text{in $L^2(\widehat{H})$}.
\]
The isomorphism $\mathcal{T}_{H,a}$ and the $\Pi$-shifting property \eqref{shifting5}  give, for any $x\in \mathcal{A}_a$, the sampling formula
\[
\begin{split}
 x&=\sum_{j=1}^s\sum_{m\in M}\mathcal{L}_jx(m)\mathcal{T}_{H,a}\big(r H_j\chi_m\big)
   =\sum_{j=1}^s\sum_{m\in M}\mathcal{L}_jx(m)\,\Pi(m)\mathcal{T}_{H,a}\big(r H_j\big)\\
   &=\sum_{j=1}^s\sum_{m\in M}\mathcal{L}_jx(m)\,\Pi(m)c_{j,h}\,,
\end{split}
\]
where $c_{j, h}:=\mathcal{T}_{H,a}(rH_j)\in \mathcal{A}_a$, $j=1,2,\dots, s$. Besides, 
$\big\{\Pi(m) c_{j,h}  \big\}_{m\in M;\,j=1,2,\ldots,s}$ is a frame for $\mathcal{A}_a$.  In fact, the following result  holds:

\begin{theorem}
\label{regular3}
For $x\in \Ac_a$ consider the sequence of samples $\{\mathcal{L}_jx(rm)\}_{m\in \mathbb{Z};\,j=1,2,\dots, s}$ defined in \eqref{samples6}. Assume that the functions $G_j$, $j=1,2,\dots,s$, in \eqref{samples7} belong to $L^\infty(\widehat{H})$, and consider the associated $\mathbb{G}(\xi)$ matrix given in \eqref{Gmatrix}. The following statements are equivalent:
\begin{enumerate}[(a)]
\item The constant $\alpha_{\mathbb{G}}>0$.
\item There exist functions  $H_j(\xi)$ in $L^\infty(\widehat{H})$, $j=1, 2, \dots, s$,  satisfying
\[
\big(H_1(\xi), H_2(\xi), \dots, H_s(\xi)\big)\, \mathbb{G}(\xi)=(1, 0, \dots, 0) \quad \text{a.e.  $\xi$ in $\widehat{H}$}\,.
\]
\item There exist $c_j\in \mathcal{A}_a$, $j=1,2, \dots, s$, such that the sequence
$\big\{\Pi(m)c_j\big\}_{m\in \mathbb{Z};\,j=1,2,\dots s}$ is a frame for $\mathcal{A}_a$, and for any $x\in \mathcal{A}_a$ the expansion
\begin{equation}
\label{sampling3}
x=\sum_{j=1}^s \sum_{m \in M}\mathcal{L}_j x(m)\,\Pi(m)c_j \quad \text{in $\mathcal{H}$}\,,
\end{equation}
holds. 
\item There exists a frame $\big\{C_{j,m}\big\}_{m\in M;\,j=1,2,\dots s}$ for $\mathcal{A}_a$ such that, for each $x\in \mathcal{A}_a$ the expansion
\[
x=\sum_{j=1}^s \sum_{m\in M}\mathcal{L}_{j} x(m)\,C_{j,m} \quad \text{in $\mathcal{H}$}\,,
\]
holds.
\end{enumerate}
In case the equivalent conditions are satisfied, for the elements $c_j$ in $(c)$ we have $c_{j}=\mathcal{T}_{H,a}(rH_j)$, for some functions $H_j$ in $L^\infty(\widehat{H})$, $j=1,2,\dots, s$, and satisfying the condition in $(b)$.
\end{theorem}
\begin{proof}
We have already proved that $(a)$ implies $(b)$, and that $(b)$ implies $(c)$. Obviously, $(c)$ implies $(d)$.
The proof concludes as that of Theorem 1 in \cite{garcia:17}.
\end{proof}

\subsubsection*{Notes and remarks}
Some specific comments for the case treated in this section are the following:
\begin{enumerate}
\item The LCA group approach is not just a unified way of dealing with the four classical  groups $\mathbb{R}, \mathbb{Z}, \mathbb{T}, \mathbb{Z}_N$: signal processing often involves products of these groups which are also LCA groups. For example, multichannel video signal involves the group $\mathbb{Z}^d \times \mathbb{Z}_N$, where $d$ is the number of channels and $N$ the number of pixels of each image. The availability of an abstract sampling theory for unitary invariant spaces becomes a useful tool to handle these problems in a unified way. Moreover, any notational complication is avoided especially in the multidimensional setting.

\item Notice that the case exhibited in this section is more general than those in the former sections. Indeed, let $G$ be an LCA group and let $H:=\{kg\}_{k\in \ZZ}$ the (infinite) cyclic group generated by some fix element $g\in G$. Then, the subspace
\[
\mathcal{A}_a:=\overline{\operatorname{span}}\big\{\Pi(kg)a\,:\, k\in \ZZ\big\}=\overline{\operatorname{span}}\big\{[\Pi(g)]^ka\,:\, k\in \ZZ\big\}\,,
\]
is obtained from $T:=\Pi(g)$.

\item The cyclid case in Section \ref{section4} involving the group $\ZZ_N$ may be a particular case of this section. Although there we have used the canonical basis of $\CC^N$ to define the  isomorphism $\mathcal{T}_{N,a}$ instead of the Fourier basis of the characters used here. So in the finite case we have not used, explicitly, Fourier analysis.

\item In case the operator $T=U$ is unitary, a necessary and sufficient condition for $\{T(h)a\}_{h\in H}$ to be a Riesz sequence in $\mathcal{H}$ can be found in \cite{barbieri:15}.

\item In the {\em overcomplete} setting we have that $s>r$. In case $r=s$, the frame condition in Theorem \ref{regular3} becomes a Riesz basis condition: There exist $r$ unique elements $c_j\in \Ac_a$, $j=1, 2, \dots, r$, such that the sequence $\big\{\Pi(m)c_j\big\}_{m\in M;\,j=1,2,\dots r}$ is a {\em Riesz basis} for 
$\mathcal{A}_a$, and the sampling expansion \eqref{sampling3} holds. Moreover, due to the uniqueness of the coefficients in a Riesz basis expansion, the {\em interpolation property} $\mathcal{L}_{j'} c_j(m)=\delta_{j,j'}\, \delta_{m,0}$, where $m\in M$ ($0$ denotes the null element in $G$) and $j, j'=1,2,\dots, r$, holds (for a similar result, see \cite[Corollary 2]{garcia:17}).

\item For the left regular unitary representation $g\in G \mapsto L_g$ of the group $G$ in $L^2(G)$, i.e., $L_gx(g')=x(g'-g)$ for any $g'\in G$ and $x\in L^2(G)$, let $\Ac_a$ be the $H$-shift-invariant subspace of $L^2(G)$ given by $\big\{\sum_{h\in H} \alpha_h\,L_ha(g)\,\,:\,\, \{\alpha_h\}\in \ell^2(H)Ê\big\}$, where $H$ denotes a uniform lattice in $G$. In this context, a Kluv\'anek's sampling theorem in $\Ac_a$ can be deduced from Theorem \ref{regular3}; see \cite[Theorem 3]{garcia:17} for the details.

\end{enumerate}

\medskip

\noindent{\bf Acknowledgments:} 
The authors wish to thank Prof. Christensen for let us know his work \cite{ole:17} on operator representations of frames.
This work has been supported by the grant MTM2017-84098-P from the Spanish {\em Ministerio de Econom\'{\i}a y Competitividad (MINECO)}.
\vspace*{0.3cm}

\end{document}